\documentclass[a4paper]{article}
\usepackage{mathrsfs}
\usepackage{latexsym,bm}
\usepackage{amssymb,amsmath,amsthm}
\usepackage[english]{babel}
\usepackage{titlesec}
\titleformat{\subsection}
{\normalfont\normalsize\itshape}
{\itshape\thesubsection}
{0.5em}
{}
\usepackage[colorlinks=true]{hyperref}
\usepackage{graphicx}
\usepackage{color}
\usepackage{appendix}
\allowdisplaybreaks[4]

\newtheorem{theorem}{Theorem}
\newtheorem{lemma}{Lemma}
\newtheorem{remark}{Remark}
\newtheorem{proposition}{Proposition}

\numberwithin{equation}{section}
\hypersetup{linkcolor=blue,urlcolor=red,citecolor=red}
\newcommand{\Om}{\Omega}
\newcommand{\om}{\omega}
\newcommand{\R}{\mathbb R}

\newcommand{\norm}[1]{\left\|#1\right\|}

\newcommand{\Hh}{\mathcal H}
\newcommand{\V}{\mathcal V}
\newcommand{\Ll}{\mathcal L}
\newcommand{\Ff}{\mathcal F}
\newcommand{\Dd}{\mathcal D}
\newcommand{\dd}{\,\mathrm d}

\newcommand{\Aop}{\mathbf A}

\title{Null controllability on measurable sets for the complex cubic Ginzburg--Landau equation}
\author{Yu Xiao\footnote{School of Mathematics and Statistics, Wuhan University, Wuhan 430072, P. R. China (xiaoyu\_math@whu.edu.cn).}\and Can Zhang\footnote{School of Mathematics and Statistics, Wuhan University, Wuhan 430072, P. R. China (canzhang@whu.edu.cn).}}
\date{}
\begin{document}
\maketitle

\begin{abstract}
We study local null controllability in arbitrarily small time for the complex cubic Ginzburg--Landau equation on space-time measurable sets. By a slicing argument and the density point property, we select shrinking time intervals. On each interval, we construct a finite-rank feedback and prove a quantitative stabilization estimate. Applying these feedbacks with increasing decay rates, we obtain local null controllability. For complex controls, we use a spectral feedback which is static on each interval. For a single real control, we construct a feedback based on a Gramian which depends on time. When the spatial control region is disjoint from a nonempty open subset, we also obtain a lower bound for the minimal time of null controllability for suitable initial data. This shows that global null controllability in arbitrarily small time fails.
\end{abstract}


\section{Introduction}

Let $\Omega\subset\mathbb R^d$ ($d=1,2,3$) be a bounded, connected open set with a smooth boundary. Let $T\in(0,1)$. Unless otherwise specified, the control region $\mathcal D$ is a Lebesgue measurable subset of $(0,T)\times\Omega$. We consider the following complex cubic Ginzburg--Landau equation with a distributed control $u$:
\begin{equation}\label{eq:CGL}
\begin{cases}
\partial_t y=(1+i\alpha)\Delta y+Qy-(1+i\beta)|y|^2y+\mathcal X_{\Dd}u&\text{in }(0,T)\times\Omega,\\
y=0&\text{on }(0,T)\times\partial\Omega,\\
y(0,\cdot)=y_0&\text{in }\Omega,
\end{cases}
\end{equation}
where $\alpha,\beta,Q\in\mathbb R$, $\mathcal X_{\Dd}$ is the characteristic function of $\Dd$, and $u(t,\cdot)\in L^2(\Omega;\mathbb K)$ for $\mathbb K\in\{\mathbb C,\mathbb R\}$.

\subsection{Motivation}

Several control problems have been studied for the complex Ginzburg--Landau equation (CGLE for short), including controllability \cite{CMM,F,FWYZ,RZ,ZZ}, time-optimal control \cite{LZ2}, and stabilization \cite{ASK}. Concerning controllability, we refer to \cite{F,FWYZ} for the global null controllability of linear equations and to \cite{CMM,RZ} for local controllability results of nonlinear equations. We also refer to \cite{ZZ} for approximate controllability in small time for the nonlinear equation with bilinear controls.

For a linear control system, the standard duality argument shows that null controllability is equivalent to an appropriate observability inequality for the adjoint system. Thus, for the linear CGLE, it suffices to establish the corresponding observability inequality. For control regions of the form $\mathcal{D}=(0,T)\times\omega$, with $\omega\subset\Omega$ nonempty and open, a classical approach is based on Carleman estimates. Suitable weights can be constructed in this setting to obtain observability for operators with a complex principal part; see \cite{F}. For nonlinear equations, a standard method is to establish observability and controllability for the system linearized along a suitable trajectory and then use a fixed point argument. We refer to \cite{FernandezCaraZuazua,PWZ} for semilinear heat equations and to \cite{RZ} for the complex Ginzburg--Landau equation.

For a measurable control set, the geometry introduces an additional difficulty. The classical Carleman method for open control regions does not apply directly. An extension is available for regions of the form $E\times\omega$, where $E\subset(0,T)$ is a measurable time set and $\omega\subset\Omega$ is nonempty and open; see, e.g., \cite{W} for heat equations and \cite{PW} for more general parabolic equations. For the heat equation, observability from spatial measurable sets of positive measure was established in \cite{AE2013}. Subsequently, observability from general measurable sets $\mathcal D\subset(0,T)\times\Omega$ with $|\mathcal D|>0$ was proved in \cite{AEWZ}. These developments use quantitative unique continuation, propagation of smallness, spectral inequalities, and the telescoping series method.

Linearizing the cubic nonlinearity in \eqref{eq:CGL} along a trajectory produces lower-order terms depending on space and time. Observability from measurable sets has been established for parabolic equations with analytic coefficients in \cite{EMZ2015,EMZ} and for heat equations with bounded real space-time potentials in \cite{WangZhang2026}. These results do not directly apply to the linearized equation of \eqref{eq:CGL}, which has a complex principal part and coupled lower-order terms.

A second issue concerns the number of control components. A complex control $u(t,\cdot)\in L^2(\Omega;\mathbb C)$ acts on both real components of the state. If the control $u(t,\cdot)\in L^2(\Omega;\mathbb R)$ is real and $y=p+iq$, then
$$\begin{cases}
	\partial_t p=\Delta p-\alpha\Delta q+Qp-(p^2+q^2)(p-\beta q)+\mathcal X_{\Dd}u,\\
	\partial_t q=\alpha\Delta p+\Delta q+Qq-(p^2+q^2)(\beta p+q).
\end{cases}$$
For the equation linearized at $y = 0$, the coupling $\alpha\Delta$ transfers control to the imaginary component when $\alpha\ne0$. Thus the problem with a real control is a coupled parabolic system with one control component. 

Controllability of coupled parabolic systems with fewer controls has been extensively studied under various structural assumptions. Early results for coupled systems of two equations with a single control were obtained in \cite{ABDK,GBPG,Guerrero}, where Carleman estimates are combined with the coupling structure to recover the uncontrolled component. For coupled systems with a diagonalizable diffusion matrix, a Kalman rank condition is given in \cite{ABDG} for controllability by a reduced number of localized controls. Subsequently, the case of diffusion matrices that are not diagonalizable was studied in \cite{FGT,LissyZuazua}. In all these works, the control regions are of the form $(0,T)\times\omega$, where $\omega\subset\Omega$ is nonempty and open.

Observability of coupled parabolic systems from a single component on measurable space-time sets has also been studied. For weakly coupled systems of two equations, such an observability inequality is established in \cite{EMZ2015} under analyticity assumptions on the boundary and the time-independent coefficients, together with a suitable coupling condition of order zero. Recently, the authors of \cite{FWYZ} considered a strongly coupled parabolic system with diffusion matrix $\left(\begin{smallmatrix}a&-b\\b&a\end{smallmatrix}\right)$, $a>0$, $b\ne0$, and proved observability from a single component. More generally, the result of \cite{FWYZ} was extended in \cite{CLB} to strongly coupled parabolic systems with a coercive diffusion matrix and an arbitrary number of components. It characterizes observability from measurable sets by a generalized Kalman rank condition.

We construct the fast controls through quantitative rapid stabilization. Given a decay rate $\lambda$, we seek a feedback for which the state decays at least as fast as $e^{-\lambda t}$, with explicit bounds for the feedback norm, the transient growth, and the admissible size of the initial state. Quantitative stabilization has been developed via backstepping \cite{CN,GHMSXZ}, the frequency--Lyapunov method \cite{X1,X2,XZ}, modal decomposition \cite{XXZMD}, and linear quadratic methods \cite{XXZLQ}; related Gramian constructions also appear in \cite{NguyenSch,NguyenKdV}. The frequency--Lyapunov method is particularly useful for parabolic systems because it treats the low and high frequencies separately. A finite-rank feedback stabilizes the low frequencies, while the dissipation of the equation controls the high frequencies. Combining these estimates in a Lyapunov function, we achieve rapid stabilization, while the spectral inequality further provides quantitative estimates.

Quantitative rapid stabilization also provides a method for proving null controllability. Applying feedback laws with increasing decay rates on a sequence of shrinking time intervals makes the state of the closed-loop system converge to zero at a finite accumulation time. This strategy was developed for one-dimensional parabolic equations in \cite{CN} and was subsequently used in \cite{X2} to recover the cost of null controllability for heat equations in small time. In a nonlinear setting, the same principle can be combined with perturbative estimates to obtain local null controllability; see \cite{X1,XXZMD}.

For the measurable control sets considered here, the shrinking intervals must be chosen according to the time geometry. Following the slicing argument and the density point construction used in observability from measurable sets \cite{AEWZ}, we first select a measurable set of times for which the spatial sections have a uniform positive lower bound on their measure. We then choose a density point and a sequence of shrinking intervals on which the active times occupy a uniformly positive fraction. The uniform lower bound on the measure of the spatial sections allows us to apply a spectral inequality with uniform constants. On the intervals selected by the density point argument, we then establish quantitative stabilization estimates and apply them successively.

\medskip\noindent\textbf{Notation.} For a Lebesgue measurable set $A$, $|A|$ denotes its Lebesgue measure in the ambient Euclidean space. We write $B_r(x_0)=\{x\in\mathbb R^d:|x-x_0|<r\}$ and $B_r=B_r(0)$. For a measurable set $\mathcal D\subset(0,T)\times\Omega$, let $\mathcal D_t=\{x\in\Omega:(t,x)\in\mathcal D\}$ denote its spatial section. We equip $H_0^1(\Omega;\mathbb K)$ with the norm $\|v\|_{H_0^1(\Omega;\mathbb K)}^2=\|v\|_{L^2(\Omega;\mathbb K)}^2+\|\nabla v\|_{L^2(\Omega;\mathbb K^d)}^2$. For Banach spaces $X$ and $Y$, $\mathcal L(X,Y)$ denotes the space of bounded $\mathbb R$-linear operators from $X$ to $Y$, and $\mathcal L(X)=\mathcal L(X,X)$.

\subsection{Preliminaries on measurable sets}

We assume $|\Dd|>0$. Since $\Omega$ can be covered by countably many balls $B_R(x_0)$ satisfying $B_{4R}(x_0)\subset\Omega$, at least one of the corresponding cylinders meets $\Dd$ in a set of positive measure. Restricting the control to this set and keeping the same notation, we may fix $x_0\in\Omega$ and $R\in(0,1)$ such that
\begin{equation}
	B_{4R}(x_0)\subset\Omega,\qquad \Dd\subset(0,T)\times B_R(x_0),\qquad |\Dd|>0.
\end{equation}
We first recall two standard facts concerning measurable sets; see \cite{AEWZ,PW}.
\begin{lemma}\label{lem:fubini}
Let $\mathcal D\subset(0,T)\times B_r(x_0)$ be measurable with $|\mathcal D|>0$. The sections $\mathcal D_t$ are measurable for almost every $t\in(0,T)$. Set $E=\{t\in(0,T):|\mathcal D_t|\ge |\mathcal D|/(2T)\}$, excluding the times at which $\mathcal D_t$ is not measurable. Then $E$ is measurable and $|E|\ge|\mathcal D|/(2|B_r|)$.
\end{lemma}
\begin{lemma}\label{lem:measurable-geometry}
Let $E\subset(0,T)$ be measurable with $|E|>0$. For every $q>1$ there exist a left density point $t_*\in E$, $r_0\in(0,t_*)$, and intervals $I_n=(t_n,t_{n+1})$ with 
\begin{equation}
 t_n=t_*-r_0q^{-n},\quad\ell_n=t_{n+1}-t_n,\qquad n\ge0,
\end{equation}
such that
\begin{equation}\label{eq:active-fraction}
\ell_{n+1}=\frac{\ell_n}{q},\qquad |E\cap I_n|\ge\frac12\ell_n.
\end{equation}
\end{lemma}

Let $0<\lambda_1\le\lambda_2\le\cdots$ be the Dirichlet eigenvalues of $-\Delta$, with eigenfunctions $\{\varphi_k\}_{k\ge1}$ forming an orthonormal basis of $L^2(\Omega;\mathbb R)$. For $\Lambda>0$, set $N(\Lambda)=\#\{k:\lambda_k\le\Lambda\}$ and define the projection onto the low frequencies on $L^2(\Omega;\mathbb R)$ by
\begin{equation}\label{eq:low-fre-pro}
P_\Lambda f=\sum_{k=1}^{N(\Lambda)}\langle f,\varphi_k\rangle_{L^2(\Omega;\mathbb R)}\varphi_k.
\end{equation}
We use the same notation for the complex-linear extension of this projection to $L^2(\Omega;\mathbb C)$:
$$P_\Lambda f=\sum_{k=1}^{N(\Lambda)}\langle \operatorname{Re} f,\varphi_k\rangle_{L^2(\Omega;\mathbb R)}\varphi_k+i\sum_{k=1}^{N(\Lambda)}\langle \operatorname{Im} f,\varphi_k\rangle_{L^2(\Omega;\mathbb R)}\varphi_k,\quad\forall f\in L^2(\Omega;\mathbb C).$$
 The following spectral inequality holds uniformly over the spatial sections selected by Lemma~\ref{lem:fubini}; see \cite{AEWZ} for its proof and \cite{LR,LZ} for related spectral inequalities.

\begin{lemma}\label{lem:spectral-measurable}
Suppose that $r\in(0,1)$ and $B_{4r}(x_0)\subset\Omega$. For every $\nu>0$ there is $C_0=C_0(\Omega,r,d,\nu)\ge1$ such that, for every measurable $\omega\subset B_r(x_0)$ with $|\omega|\ge\nu$, every $\Lambda>0$, and every $f\in L^2(\Omega;\mathbb R)$,
\begin{equation}\label{eq:spectral-measurable} \|P_\Lambda f\|_{L^2(\Omega;\mathbb R)}^2\le C_0e^{C_0\sqrt\Lambda}\|\mathcal X_\omega P_\Lambda f\|_{L^2(\Omega;\mathbb R)}^2. \end{equation}
\end{lemma}

We fix $\nu=|\mathcal D|/(2T)$ and the corresponding constant $C_0=C_0(\Omega,R,d,|\Dd|/(2T))$. By Lemma \ref{lem:fubini}, we have $|E|\ge |\Dd|/(2|B_R|)$. Consequently, Lemma \ref{lem:spectral-measurable} gives, for almost every $t$,
\begin{equation} \mathcal X_E(t)\|\mathcal X_{\mathcal D_t}P_\Lambda y\|_{L^2(\Omega;\mathbb R)}^2\ge C_0^{-1}e^{-C_0\sqrt\Lambda}\mathcal X_E(t)\|P_\Lambda y\|_{L^2(\Omega;\mathbb R)}^2,\quad\forall y\in L^2(\Omega;\mathbb R). \end{equation}

\subsection{Main results}

We present two feedback constructions for the low frequencies. We begin with complex controls, which act on both the real and imaginary components. In this case, we construct a simple feedback following the construction of \cite{X2}.

\begin{theorem}\label{thm:complex-control}
Let $T\in(0,1)$. There exist $q>1$, intervals $I_n=(t_n,t_{n+1})$ chosen as in Lemma~\ref{lem:measurable-geometry} with $0<t_0<t_*<T$, constants $\delta_{\mathbb C},M_{\mathbb C},c_{\mathbb C}>0$, and a strongly measurable family of finite-rank feedback operators $\mathcal K_{\Dd}^{\mathbb C}:(0,T)\to\mathcal L(L^2(\Omega;\mathbb C))$ such that
\begin{equation}
	\mathcal K_{\Dd}^{\mathbb C}\in L^\infty_{\rm loc}\bigl([t_0,t_*);\mathcal L(L^2(\Omega;\mathbb C))\bigr), \qquad \mathcal K_{\Dd}^{\mathbb C}(t)=0 \quad\text{for }t\in(0,t_0)\cup[t_*,T).
\end{equation}
More precisely, one can choose a constant $A_{\mathbb C}>0$, with $\Lambda_n=A_{\mathbb C}^2/\ell_n^2$, such that
\begin{equation}
	\mathcal K_{\Dd}^{\mathbb C}(t)y=-C_0\Lambda_n e^{C_0\sqrt{\Lambda_n}}P_{\Lambda_n}y,\quad \forall t\in I_n.
\end{equation}
For every $y_0\in H_0^1(\Omega;\mathbb C)$ with $\|y_0\|_{H_0^1(\Omega;\mathbb C)}\le\delta_{\mathbb C}$, the equation \eqref{eq:CGL} with $u(t)=\mathcal X_E(t)\mathcal K_{\Dd}^{\mathbb C}(t)y(t)$ admits a unique solution
$$y\in  C([0,T];H_0^1(\Omega;\mathbb C))
\cap L^2(0,T;H^2(\Omega;\mathbb C)\cap H_0^1(\Omega;\mathbb C))
\cap H^1(0,T;L^2(\Omega;\mathbb C)),$$
verifying
\begin{equation} y(t)=0,\quad t_*\le t\le T.\end{equation}
Moreover, for almost every $t\in(t_0,t_*)$,
\begin{equation}
\|y(t)\|_{H_0^1(\Omega;\mathbb C)}+\|\mathcal X_{\Dd_t}u(t)\|_{L^2(\Omega;\mathbb C)}\le M_{\mathbb C}\exp\!\left(-\frac{c_{\mathbb C}}{t_*-t}\right)\|y_0\|_{H_0^1(\Omega;\mathbb C)},
\end{equation}
and
\begin{equation}
\|\mathcal X_{\Dd}u\|_{L^\infty(0,T;L^2(\Omega;\mathbb C))}\le M_{\mathbb C}\|y_0\|_{H_0^1(\Omega;\mathbb C)}.
\end{equation}
\end{theorem}

The feedback construction in Theorem~\ref{thm:complex-control} does not apply to a single real control, which acts only on the real component. In Theorem~\ref{thm:real-control}, we use the coupling $\alpha\Delta$ to construct a feedback which depends on time. This construction is based on a periodic Gramian and an infinite-horizon linear-quadratic control problem.

\begin{theorem}\label{thm:real-control}
Let $T\in(0,1)$ and $\alpha\neq0$. There exist $q>1$, intervals $I_n=(t_n,t_{n+1})$ chosen as in Lemma~\ref{lem:measurable-geometry} with $0<t_0<t_*<T$, constants $\delta_{\mathbb R},M_{\mathbb R},c_{\mathbb R}>0$, and a strongly measurable family of finite-rank feedback operators $\mathcal K_{\Dd}^{\mathbb R}(t):L^2(\Omega;\mathbb C)\to L^2(\Omega;\mathbb R)$ satisfying 
$$\mathcal K_{\Dd}^{\mathbb R}\in L^\infty_{\rm loc}([t_0,t_*);\mathcal L(L^2(\Omega;\mathbb C),L^2(\Omega;\mathbb R))),\qquad \mathcal K_{\Dd}^{\mathbb R}(t)=0 \quad\text{for }t\in (0,t_0)\cup[t_*,T).$$
More precisely, one can choose a constant $A_{\mathbb R}>0$, with $\Lambda_n=Q+15A_{\mathbb R}^2/(16\ell_n^2)$, such that
\begin{equation}
\mathcal K_{\Dd}^{\mathbb R}(t)y=-\sum_{j=1}^{N(\Lambda_n)}\bigl[B_n(t)^\top W_n(t-t_n)^{-1}Z_n(y)\bigr]_j\varphi_j,\quad \forall t\in I_n.
\end{equation}
Here $Z_n(y)$ contains the low-frequency modes of the real and imaginary parts of $y$, $B_n(t)$ is the input matrix, and $W_n$ is the periodic Gramian; see Section~\ref{sec:prelim}. For every $y_0\in H_0^1(\Omega;\mathbb C)$ with $\|y_0\|_{H_0^1(\Omega;\mathbb C)}\le\delta_{\mathbb R}$, the equation \eqref{eq:CGL} with the real-valued feedback control $u(t)=\mathcal K_{\Dd}^{\mathbb R}(t)y(t)$ admits a unique solution
$$y\in C([0,T];H_0^1(\Omega;\mathbb C))
\cap L^2(0,T;H^2(\Omega;\mathbb C)\cap H_0^1(\Omega;\mathbb C))
\cap H^1(0,T;L^2(\Omega;\mathbb C)),$$
verifying
\begin{equation}y(t)=0,\quad t_*\le t\le T.\end{equation}
Moreover, for almost every $t\in(t_0,t_*)$,
\begin{equation}
\|y(t)\|_{H_0^1(\Omega;\mathbb C)}+\|\mathcal X_{\Dd_t}u(t)\|_{L^2(\Omega;\mathbb R)}\le M_{\mathbb R}\exp\!\left(-\frac{c_{\mathbb R}}{t_*-t}\right)\|y_0\|_{H_0^1(\Omega;\mathbb C)},
\end{equation}
and
\begin{equation}
\|\mathcal X_{\Dd}u\|_{L^\infty(0,T;L^2(\Omega;\mathbb R))}\le M_{\mathbb R}\|y_0\|_{H_0^1(\Omega;\mathbb C)}.
\end{equation}
\end{theorem}

We make the following remarks.
\begin{remark}
The two theorems provide different feedback laws. When $\alpha\ne0$, the real-valued feedback in Theorem~\ref{thm:real-control} also gives an admissible complex-valued control. The feedback in Theorem~\ref{thm:complex-control} is also of independent interest. It is explicit, is static on each interval, and remains available when $\alpha=0$. Moreover, for $\Dd=(0,T)\times\Omega$, each low-frequency mode of the linear closed-loop system satisfies
$$\dot y_j=(Q-\lambda_j-\mu_\Lambda-i\alpha\lambda_j)y_j,\qquad 1\le j\le N(\Lambda).$$
Thus the feedback shifts the real part of each low-frequency eigenvalue by $-\mu_\Lambda$. For sufficiently large $\Lambda$, these eigenvalues lie in $\{z\in\mathbb C:\operatorname{Re}z\le-\Lambda\}$. The two constructions use different Lyapunov functions, namely the static physical energy in \eqref{eq:Lya} and the Gramian energy depending on time in \eqref{eq:Vn}.
\end{remark}
\begin{remark}
The feedback families are locally essentially bounded before $t_*$. More precisely, for each $\mathbb K\in\{\mathbb C,\mathbb R\}$ there exists $A>0$ such that, for every $\varepsilon\in(0,t_*-t_0)$,
\begin{equation*} \|\mathcal K_{\Dd}^{\mathbb K}\|_{L^\infty(t_0,t_*-\varepsilon;\mathcal L(L^2(\Omega;\mathbb C),L^2(\Omega;\mathbb K)))}\le A\exp(A/\varepsilon). \end{equation*}
\end{remark}
\begin{remark}\label{rem:cylinder-both}
For cylindrical control regions $\Dd=(0,T)\times\omega$, where $\omega\subset\Omega$ is measurable with $|\omega|>0$, the constants in Theorems~\ref{thm:complex-control} and~\ref{thm:real-control} can be chosen for the respective control classes so that
\begin{equation}\label{eq:cylinder-cost}
 \delta_{\mathbb K}\ge C_{\mathbb K}^{-1}e^{-C_{\mathbb K}/T},\qquad M_{\mathbb K}\le C_{\mathbb K}e^{C_{\mathbb K}/T},
\end{equation}
where $C_{\mathbb K}$ is independent of $T$. 
\end{remark}
\begin{remark}\label{rem:alpha-zero}
Assume $\alpha=0$. There do not exist constants $\delta,C>0$ such that every $y_0\in H_0^1(\Omega;\mathbb C)$ with $\|y_0\|_{H_0^1(\Omega;\mathbb C)}\le\delta$ admits a mild solution in $C([0,T];H_0^1(\Omega;\mathbb C))$ driven to zero at time $T$ by a real control satisfying
\begin{equation}\label{eq:impossible-linear-cost}
\|\mathcal X_{\Dd}u\|_{L^\infty(0,T;L^2(\Omega;\mathbb R))}\le C\|y_0\|_{H_0^1(\Omega;\mathbb C)}.
\end{equation}
\end{remark}
\vspace{2mm}

Theorems~\ref{thm:complex-control} and~\ref{thm:real-control} concern sufficiently small initial data. It is then natural to ask whether global null controllability holds in arbitrarily small time. Let $\omega\subset\Omega$ be measurable with $|\omega|>0$, and set $\Dd=(0,s)\times\omega$ for $s>0$. For $\mathbb K\in\{\mathbb R,\mathbb C\}$, denote by $\tau_{\mathbb K}(y_0)$ the infimum of times $s>0$ for which a control $u\in L^2((0,s)\times\Omega;\mathbb K)$, acting through $\mathcal X_\omega u$ in \eqref{eq:CGL}, drives the state to zero through an admissible strong solution
$$y\in C([0,s];H_0^1(\Omega;\mathbb C))
\cap L^2(0,s;H^2(\Omega;\mathbb C)\cap H_0^1(\Omega;\mathbb C))
\cap H^1(0,s;L^2(\Omega;\mathbb C)).$$
Set $\inf\varnothing=+\infty$ and
$$\mathcal T_{\mathbb K}(R)=\sup_{y_0\in H_0^1(\Omega;\mathbb C),\ \|y_0\|_{H_0^1(\Omega;\mathbb C)}\le R}\tau_{\mathbb K}(y_0).$$

\begin{theorem}[Minimal null-control times]\label{thm:minimal-time}
Assume that there exist $x_*\in\Omega$ and $r_*>0$ such that \begin{equation}\label{eq:fixed-hole}B_{4r_*}(x_*)\subset\Omega \;\;\text{and}\;\;|\omega\cap B_{4r_*}(x_*)|=0.\end{equation}
 Then there are constants $c,R_0>0$ and $\phi\in C_c^\infty(B_{r_*}(x_*);\mathbb R)$ with $\|\phi\|_{H_0^1(\Omega;\mathbb R)}=1,$ such that, for every $0<R<R_0$ and $\mathbb K\in\{\mathbb C,\mathbb R\}$,
\begin{equation}
\tau_{\mathbb K}(R\phi)\ge\frac{c}{|\log R|}.
\end{equation}
Moreover, for $\mathbb K=\mathbb C$, and also for $\mathbb K=\mathbb R$ when $\alpha\ne0$, there exists $C>0$ such that
\begin{equation}\label{eq:tau-two-sided}
\frac{c}{|\log R|}\le\tau_{\mathbb K}(R\phi)\le\mathcal T_{\mathbb K}(R)\le\frac{C}{|\log R|}.
\end{equation}
\end{theorem}

\begin{remark}
For dissipative heat equations with locally Lipschitz, nondecreasing nonlinearities satisfying $f(0)=0$ and $f(s)\ge cs(\log s)^p$ for sufficiently large $s$ and some $p>2$,  the authors of \cite{AnitaTataru} show that global null controllability in arbitrarily small time fails for a nonempty open subset $\omega$ satisfying $\Omega\setminus\overline\omega\ne\varnothing$, and obtain lower bounds of order $C/|\log R|$ for the minimal time of null controllability on small $L^2$ balls. The same logarithmic order for the minimal time of null controllability on small $L^\infty$ balls was obtained for viscous Hamilton--Jacobi equations; see \cite{PorrettaZuazua}. Remark~7.7 in \cite{Coron} also recalls a related negative result for the cubic heat equation with boundary control, where reaching a nonzero target from zero requires a positive time. In this paper, Theorem~\ref{thm:minimal-time} gives bounds for the minimal time of null controllability for the complex cubic Ginzburg--Landau equation under \eqref{eq:fixed-hole}, with the same order as in \cite{AnitaTataru}. Note that the above geometric condition on $\omega$ in \cite{AnitaTataru} also implies \eqref{eq:fixed-hole}.
\end{remark}

\begin{remark}
We emphasize a feature of the present work. In this paper, we give a constructive proof of null controllability on measurable sets by combining the slicing and density point argument for measurable sets with quantitative rapid stabilization and an iteration in time. This provides an alternative to the linearization approach for establishing controllability of nonlinear equations on measurable space-time sets. For the cubic CGLE, this approach handles the nonlinearity and the coupling between components, and provides feedbacks of finite rank on each time interval, even with a single real control when $\alpha\ne0$. The method may also extend to more general semilinear and quasilinear parabolic equations; see \cite{KLZ2026} for recent quantitative controllability results for quasilinear parabolic equations.
\end{remark}

\noindent\textbf{Organization.} Section~\ref{sec:con-wepos} establishes local well-posedness for the complex and real formulations. Sections~\ref{sec:qrs} and~\ref{sec:prelim} establish the stabilization estimates on each interval for the spectral feedback and the Gramian feedback, respectively. Using these estimates, Section~\ref{sec:iteration} proves Theorems~\ref{thm:complex-control} and~\ref{thm:real-control} by an iteration in time. The proofs of Remarks~\ref{rem:cylinder-both} and \ref{rem:alpha-zero} are also given in this section. Section~\ref{sec:minimal-time} studies the minimal time of null controllability and proves Theorem \ref{thm:minimal-time}.

\medskip\noindent\textbf{Acknowledgments.} This work was supported by the National Natural Science Foundation of China under Grant No.~12422118.

\section{Well-posedness with feedback}\label{sec:con-wepos}
We first establish local well-posedness for the closed-loop system with a complex feedback.

\begin{lemma}\label{lem:we-pos-th}
Let $\mathcal D\subset[0,+\infty)\times\Omega$ be measurable, and fix $R>0$, $\Lambda>0$, $\mu\ge0$, and $t_1\ge0$. There exists $\widetilde T>0$, depending only on $\Omega,d,Q,\beta,R,\mu$, such that, for every $y_0\in H_0^1(\Omega;\mathbb C)$ with $\|y_0\|_{H_0^1(\Omega;\mathbb C)}\le R$, the problem
\begin{equation} \begin{cases} \partial_t y=(1+i\alpha)\Delta y+Qy-(1+i\beta)|y|^2y-\mu\mathcal X_{\mathcal D}P_\Lambda y&\text{in }(t_1,t_2)\times\Omega,\\ y=0&\text{on }(t_1,t_2)\times\partial\Omega,\\ y(t_1,\cdot)=y_0&\text{in }\Omega \end{cases} \end{equation}
has a unique solution
$$y\in C([t_1,t_2];H_0^1(\Omega;\mathbb C))\cap L^2(t_1,t_2;H^2(\Omega;\mathbb C)\cap H_0^1(\Omega;\mathbb C))\cap H^1(t_1,t_2;L^2(\Omega;\mathbb C)),$$
whenever $0<t_2-t_1<\widetilde T$. If the maximal existence interval is $[t_1,t_1+T_{\max})$ with $T_{\max}<\infty$, then
$$\limsup_{t\uparrow t_1+T_{\max}}\|y(t)\|_{H_0^1(\Omega;\mathbb C)}=+\infty.$$
\end{lemma}

\begin{proof}
Fix $0<\tau\le1$ and put
$$I_\tau=(t_1,t_1+\tau),\qquad Y_\tau=C([t_1,t_1+\tau];H_0^1(\Omega;\mathbb C))\cap L^2(I_\tau;H^2(\Omega;\mathbb C)\cap H_0^1(\Omega;\mathbb C)).$$
Let $z$ solve the linear Dirichlet problem
\begin{equation}\label{eq:linear-eq1}
\begin{cases}
\partial_t z-(1+i\alpha)\Delta z=f&\text{in }I_\tau\times\Omega,\\
z=0&\text{on }I_\tau\times\partial\Omega,\\
z(t_1,\cdot)=z_0&\text{in }\Omega.
\end{cases}
\end{equation}
By the standard regularity estimate for complex parabolic equations (see, for instance, \cite[Lemmas~6.1--6.2]{RZ}), there exists $D_1\ge1$, independent of $t_1$, $\alpha$ and $0<\tau\le1$, such that
\begin{equation}\label{eq:linear-RZ}
\|z\|_{Y_\tau}\le D_1\bigl(\|z_0\|_{H_0^1(\Omega;\mathbb C)}+\|f\|_{L^2(I_\tau;L^2(\Omega;\mathbb C))}\bigr).
\end{equation}

Since $d\le3$, the Sobolev embedding $H_0^1(\Omega;\mathbb C)\hookrightarrow L^6(\Omega;\mathbb C)$ (see \cite{Brezis}) gives
\begin{equation}\label{eq:sobolev-H1-L6}
\|v\|_{L^6(\Omega;\mathbb C)}\le C_S\|v\|_{H_0^1(\Omega;\mathbb C)},\quad\forall v\in H_0^1(\Omega;\mathbb C).
\end{equation}
The pointwise inequality \begin{equation}\label{eq:POINTSWISE}\bigl||z_1|^2z_1-|z_2|^2z_2\bigr|\le C(|z_1|^2+|z_2|^2)|z_1-z_2|\end{equation} yields
\begin{equation}\label{eq:wp-cubic-diff}
\||v|^2v-|w|^2w\|_{L^2(I_\tau;L^2(\Omega;\mathbb C))}\le C\tau^{1/2}(\|v\|_{Y_\tau}^2+\|w\|_{Y_\tau}^2)\|v-w\|_{Y_\tau}.
\end{equation}

For $v\in Y_\tau$, let $\Gamma(v)$ denote the solution of \eqref{eq:linear-eq1} with initial datum $y_0$ and right-hand side
$Qv-(1+i\beta)|v|^2v-\mu\mathcal X_{\mathcal D}P_\Lambda v.$
It follows from \eqref{eq:sobolev-H1-L6} that, for some $C>0$,
\begin{equation}\label{eq:wp-map-short}
\|\Gamma(v)\|_{Y_\tau}\le D_1\|y_0\|_{H_0^1(\Omega;\mathbb C)}+C\tau^{1/2}\bigl((|Q|+\mu)\|v\|_{Y_\tau}+(1+|\beta|)\|v\|_{Y_\tau}^3\bigr).
\end{equation}
Set $R_*=2D_1R$ and $\mathcal B_{R_*}=\{v\in Y_\tau:\|v\|_{Y_\tau}\le R_*\}$. By \eqref{eq:wp-map-short}, choosing $\tau>0$ so that $C\tau^{1/2}(|Q|+\mu+2(1+|\beta|)R_*^2)\le1/2$ gives $\Gamma(\mathcal B_{R_*})\subset\mathcal B_{R_*}$. Moreover, \eqref{eq:linear-RZ} and \eqref{eq:wp-cubic-diff} yield, for $v,w\in\mathcal B_{R_*}$,
$$\|\Gamma(v)-\Gamma(w)\|_{Y_\tau}\le C\tau^{1/2}\bigl(|Q|+\mu+2(1+|\beta|)R_*^2\bigr)\|v-w\|_{Y_\tau}.$$
After reducing $\tau$ if necessary, $\Gamma$ is a strict contraction. Banach's fixed point theorem then gives the required solution. Moreover, the equation implies $\partial_t y\in L^2(I_\tau;L^2(\Omega;\mathbb C))$.

For two solutions in $Y_\tau$, the boundedness of their $H_0^1$ norms allows us to apply the same difference estimate on successive short subintervals. This proves uniqueness. The continuation criterion follows from the standard continuation argument.
\end{proof}

We next establish local well-posedness for the closed-loop system with a real feedback. To this end, we identify $y=p+iq$ with $X=(p,q)^\top$ and set
$$\Hh=L^2(\Omega;\mathbb R^2),\qquad \V=H_0^1(\Omega;\mathbb R^2),\qquad \Aop=-\operatorname{diag}(\Delta,\Delta),\qquad D(\Aop)=H^2(\Omega;\mathbb R^2)\cap\V.$$
We introduce the skew-symmetric matrix
$$
 J=\begin{pmatrix}0&-1\\ 1&0\end{pmatrix},\qquad J^\top=-J.
$$
In this notation, the real formulation of \eqref{eq:CGL} is
$$\partial_t X=\Ll X+B_{\Dd_t}u+\Ff(X),\qquad \Ll=-(I+\alpha J)\Aop+QI,\qquad \Ff(X)=-(I+\beta J)|X|^2X,$$
where $B_\om u=(\mathcal X_\om u,0)^\top$. The skew-symmetry of $J$ gives, for $X\in D(\Aop)$,
\begin{equation}\label{eq:linear-energy}
 \langle \Aop X,J\Aop X\rangle_{\Hh}=0,\qquad
 2\langle \Aop X,\Ll X\rangle_{\Hh}
 =-2\|\Aop X\|_{\Hh}^2+2Q\|\Aop^{1/2}X\|_{\Hh}^2.
\end{equation}

Using \eqref{eq:sobolev-H1-L6} and \eqref{eq:POINTSWISE}, we have
\begin{equation}\label{eq:F-cubic}
 \|\Ff(X)\|_{\Hh}\le C\|X\|_{\V}^3,
\end{equation}
and
\begin{equation}\label{eq:F-cubic-diff}\|\Ff(X)-\Ff(Y)\|_{\Hh} \le C(\|X\|_{\V}^2+\|Y\|_{\V}^2)\|X-Y\|_{\V}.\end{equation}
Arguing as in Lemma~\ref{lem:we-pos-th}, we obtain the following local well-posedness result for the real formulation.
\begin{lemma}\label{lem:wellposedness}
Let $I=(a,b)$ be a finite interval and let $K\in L^{\infty}(I;\mathcal L(\V,L^2(\Omega;\mathbb R)))$. For every $X_a\in\V$, the closed-loop equation
$$\partial_t X=\Ll X+B_{\Dd_t}K(t)X+\Ff(X),\qquad X(a)=X_a,$$
admits a unique local mild solution $X$, continuous with values in $\V$, that extends to a maximal interval $[a,\tau_{\max})$ with $\tau_{\max}\in(a,b]$ (or $[a,b]$ if it does not blow up). For every $c\in(a,\tau_{\max})$, this solution belongs to $L^2(a,c;D(\Aop))\cap H^1(a,c;\Hh)$. If $\sup_{a\le t<\tau_{\max}}\|X(t)\|_{\V}<\infty$, then $\tau_{\max}=b$ and $X$ extends continuously in $\V$ to $b$.
\end{lemma}

We shall also use the following estimates when the initial datum and the forcing term are small.
\begin{lemma}\label{lem:appex-1}
There exist $\epsilon\in(0,1)$ and $C_f\ge1$, depending only on $\Omega,d,Q,\beta$, such that, for every $0<T\le1$, $y_0\in H_0^1(\Omega;\mathbb C)$ and $f\in L^2(0,T;L^2(\Omega;\mathbb C))$ with
$$\eta:=\|y_0\|_{H_0^1(\Omega;\mathbb C)}+\|f\|_{L^2(0,T;L^2(\Omega;\mathbb C))}\le\epsilon,$$
the Dirichlet problem
$$\partial_t y=(1+i\alpha)\Delta y+Qy-(1+i\beta)|y|^2y+f,\qquad y|_{\partial\Omega}=0,\qquad y(0)=y_0,$$
has a unique strong solution on $[0,T]$ satisfying
\begin{equation}\label{eq:CCGLE-appex-y}
\|y\|_{C([0,T];H_0^1(\Omega;\mathbb C))}\le C_f\eta.
\end{equation}
Moreover, there exist $\epsilon_2\in(0,\epsilon]$ and $C_2\ge1$, depending only on $\Omega,d,\alpha,\beta,Q$, such that, if $f=0$, $y_0\in H^2(\Omega;\mathbb C)\cap H_0^1(\Omega;\mathbb C)$ and $\|y_0\|_{H^2(\Omega;\mathbb C)}\le\epsilon_2$, then
\begin{equation}\label{eq:small-data-H2}
\|y\|_{C([0,T];H^2(\Omega;\mathbb C))}
+\|\partial_t y\|_{C([0,T];L^2(\Omega;\mathbb C))}
\le C_2\|y_0\|_{H^2(\Omega;\mathbb C)}.
\end{equation}
\end{lemma}

\begin{proof}
Set
$$Y_T=C([0,T];H_0^1(\Omega;\mathbb C))\cap L^2(0,T;H^2(\Omega;\mathbb C)\cap H_0^1(\Omega;\mathbb C)).$$
Applying \eqref{eq:linear-RZ} to $w(t)=e^{-Qt}z(t)$ gives, for the solution $z$ to $\partial_t z-(1+i\alpha)\Delta z-Qz=g$,
\begin{equation}\label{eq:uncontrolled-linear}
\|z\|_{Y_T}\le D\bigl(\|z_0\|_{H_0^1(\Omega;\mathbb C)}+\|g\|_{L^2(0,T;L^2(\Omega;\mathbb C))}\bigr),
\end{equation}
where $D\ge1$ is independent of $\alpha$ and $0<T\le1$. Let $\Gamma(v)$ solve this linear equation with initial datum $y_0$ and forcing $g=f-(1+i\beta)|v|^2v$. Arguing as in the proof of \eqref{eq:wp-cubic-diff}, we find
$$\|\Gamma(v)\|_{Y_T}\le D\eta+C\|v\|_{Y_T}^3,\qquad
\|\Gamma(v)-\Gamma(w)\|_{Y_T}\le C(\|v\|_{Y_T}^2+\|w\|_{Y_T}^2)\|v-w\|_{Y_T}.$$
For $\eta\le\epsilon$ with $\epsilon$ sufficiently small, $\Gamma$ is a strict contraction on the closed ball of radius $2D\eta$ in $Y_T$. Banach's fixed point theorem gives a solution satisfying \eqref{eq:CCGLE-appex-y} with $C_f=2D$. Every mild solution in $C([0,T];H_0^1(\Omega;\mathbb C))$ belongs to $Y_T$ by \eqref{eq:F-cubic} and \eqref{eq:uncontrolled-linear}, and uniqueness follows as in Lemma~\ref{lem:we-pos-th}.

For the second assertion, let $X=H^2(\Omega;\mathbb C)\cap H_0^1(\Omega;\mathbb C)$. Since $H^2(\Omega;\mathbb C)$ is an algebra for $d\le3$, it follows that
$$\||v|^2v-|w|^2w\|_X\le C(\|v\|_X^2+\|w\|_X^2)\|v-w\|_X.$$
Note that the semigroup generated by $A=(1+i\alpha)\Delta_D+Q$ is bounded on $X$ for $0\le t\le1$. Thus, with $f=0$, the same contraction argument in $C([0,T];X)$ gives a unique solution satisfying $\|y\|_{C([0,T];X)}\le C\|y_0\|_X$ for $\|y_0\|_X$ sufficiently small. By uniqueness, this is the solution constructed above. The equation and \eqref{eq:F-cubic-diff} then give $\partial_t y\in C([0,T];L^2(\Omega;\mathbb C))$ and \eqref{eq:small-data-H2}.
\end{proof}

\section{Stabilization on one interval with a complex control}\label{sec:qrs}
Recalling the definition of $P_\Lambda$ in \eqref{eq:low-fre-pro}, we introduce the high frequencies by $P_\Lambda^\perp=I-P_\Lambda$. For $\Lambda>1$, we consider the feedback operator
\begin{equation}\mathcal K_\Lambda=-\mu_\Lambda P_\Lambda.
\end{equation}
Choose
\begin{equation}\label{eq:para-selec}
\mu_\Lambda=C_0e^{C_0\sqrt\Lambda}\Lambda,\qquad\gamma_\Lambda=C_0^2e^{2C_0\sqrt\Lambda}\Lambda.
\end{equation}
Then $\mu_\Lambda^2=\Lambda\gamma_\Lambda$ and $\gamma_\Lambda\ge\Lambda>1$. Define the Lyapunov function
\begin{equation}\label{eq:Lya}
V_\Lambda(y)=\gamma_\Lambda\|P_\Lambda y\|_{L^2(\Omega;\mathbb C)}^2+\|\nabla P_\Lambda^\perp y\|_{L^2(\Omega;\mathbb C^d)}^2,\quad \forall y\in H_0^1(\Omega;\mathbb C).
\end{equation}

\begin{lemma}
For every $y\in H_0^1(\Omega;\mathbb C)$,
\begin{equation}\label{eq:norm-equ}
\frac12\|y\|_{H_0^1(\Omega;\mathbb C)}^2\le V_\Lambda(y)\le\gamma_\Lambda\|y\|_{H_0^1(\Omega;\mathbb C)}^2.
\end{equation}
\end{lemma}

\begin{proof}
Write $y=y_1+iy_2$. The Dirichlet spectral expansion gives, for $j=1,2$,
$$\|\nabla P_\Lambda^\perp y_j\|_{L^2(\Omega;\mathbb R^d)}^2=\sum_{k>N(\Lambda)}\lambda_k|\langle y_j,\varphi_k\rangle_{L^2(\Omega;\mathbb R)}|^2\ge\Lambda\|P_\Lambda^\perp y_j\|_{L^2(\Omega;\mathbb R)}^2,$$
while
$$\gamma_\Lambda\|P_\Lambda y_j\|_{L^2(\Omega;\mathbb R)}^2\ge\frac{\gamma_\Lambda}{\Lambda}\|\nabla P_\Lambda y_j\|_{L^2(\Omega;\mathbb R^d)}^2\ge\|\nabla P_\Lambda y_j\|_{L^2(\Omega;\mathbb R^d)}^2.$$
Since $\gamma_\Lambda\ge \Lambda \ge 1$, we have $V_\Lambda(y)\le \gamma_\Lambda \|y\|^2_{H_0^1(\Omega;\mathbb C)}$. On the other hand,
\begin{align*}
	V_\Lambda(y) &\ge  \frac{\Lambda}{2}\| P_{\Lambda} y \|_{L^2(\Omega;\mathbb C)}^2 +\frac{1}{2} \|\nabla P_{\Lambda} y \|_{L^2(\Omega;\mathbb C^d)}^2+\frac{\Lambda}{2} \| P_{\Lambda}^\perp y \|_{L^2(\Omega;\mathbb C)}^2 +\frac{1}{2} \|\nabla P_{\Lambda}^\perp y\|_{L^2(\Omega;\mathbb C^d)}^2\ge \frac{1}{2}  \|y\|^2_{H_0^1(\Omega;\mathbb C)}.
\end{align*}
This completes the proof.
\end{proof}

\begin{proposition}\label{prop:ra-sta-new}
There exists $C=C(\Omega,\Dd,T,R,d,\beta)\ge1$ such that, for every $\Lambda>1$, the finite-rank feedback $\mathcal K_\Lambda=-\mu_\Lambda P_\Lambda$ satisfies
\begin{equation}\label{eq:feedback-lambda-bound}
\|\mathcal K_\Lambda\|_{\mathcal L(L^2(\Omega;\mathbb C))}\le Ce^{C\sqrt\Lambda}.
\end{equation}
For every $n\ge0$, if
$$\|y(t_n)\|_{H_0^1(\Omega;\mathbb C)}\le\rho_{\Lambda,T}:=e^{-|Q|T}e^{-C\sqrt\Lambda},$$
then the Dirichlet problem
\begin{equation*}
\partial_t y=(1+i\alpha)\Delta y+Qy-(1+i\beta)|y|^2y+\mathcal X_E(t)\mathcal X_{\Dd}\mathcal K_\Lambda y
\end{equation*}
with initial datum $y(t_n)$ admits a unique strong solution. With $u(t)=\mathcal X_E(t)\mathcal K_\Lambda y(t),\;t\in I_n$, it satisfies
\begin{equation}\label{eq:y-decay-meas-1}
\|u(t)\|_{L^2(\Omega;\mathbb C)}+\|y(t)\|_{H_0^1(\Omega;\mathbb C)}\le e^{|Q|T}e^{C\sqrt\Lambda}\|y(t_n)\|_{H_0^1(\Omega;\mathbb C)},\quad\text{for a.e. }t\in I_n,
\end{equation}
and
\begin{equation}\label{eq:y-decay-meas-2}
\|y(t_{n+1})\|_{H_0^1(\Omega;\mathbb C)}\le e^{|Q|T}e^{C\sqrt\Lambda}e^{-\Lambda\ell_n/8}\|y(t_n)\|_{H_0^1(\Omega;\mathbb C)}.
\end{equation}
\end{proposition}

\begin{proof}
Let $y=y_1+iy_2$ and put $f=-(1+i\beta)|y|^2y=f_1+if_2$. On the interval $I_n$, the closed-loop system reads
\begin{equation}\label{eq:complex-components}
\begin{cases}
\partial_t y_1=\Delta y_1-\alpha\Delta y_2+Qy_1+f_1-\mu_\Lambda\mathcal X_E(t)\mathcal X_{\Dd_t}P_\Lambda y_1,\\
\partial_t y_2=\alpha\Delta y_1+\Delta y_2+Qy_2+f_2-\mu_\Lambda\mathcal X_E(t)\mathcal X_{\Dd_t}P_\Lambda y_2.
\end{cases}
\end{equation}
By Lemma~\ref{lem:we-pos-th}, the closed-loop system has a unique maximal solution. On every compact subinterval $J$ of its existence interval, this solution satisfies
$$y\in C(J;H_0^1(\Omega;\mathbb C))\cap L^2(J;H^2(\Omega;\mathbb C)\cap H_0^1(\Omega;\mathbb C)),\qquad \partial_t y\in L^2(J;L^2(\Omega;\mathbb C)).$$
The following energy identities therefore hold for almost every $t\in J$.

First, we test the two equations in \eqref{eq:complex-components} against $P_\Lambda y_1$ and $P_\Lambda y_2$, respectively, and obtain
\begin{equation*}
\begin{aligned}
\frac12\frac{d}{dt}\sum_{j=1}^2\|P_\Lambda y_j\|_{L^2(\Omega;\mathbb R)}^2
={}&-\sum_{j=1}^2\|\nabla P_\Lambda y_j\|_{L^2(\Omega;\mathbb R^d)}^2
+Q\sum_{j=1}^2\|P_\Lambda y_j\|_{L^2(\Omega;\mathbb R)}^2\\
&+\sum_{j=1}^2\langle f_j,P_\Lambda y_j\rangle_{L^2(\Omega;\mathbb R)}
-\mu_\Lambda\mathcal X_E(t)\sum_{j=1}^2\|\mathcal X_{\Dd_t}P_\Lambda y_j\|_{L^2(\Omega;\mathbb R)}^2.
\end{aligned}
\end{equation*}
Similarly, testing against $-\Delta P_\Lambda^\perp y_j$ gives
\begin{equation*}
\begin{aligned}
\frac12\frac{d}{dt}\sum_{j=1}^2\|\nabla P_\Lambda^\perp y_j\|_{L^2(\Omega;\mathbb R^d)}^2
={}&-\sum_{j=1}^2\|\Delta P_\Lambda^\perp y_j\|_{L^2(\Omega;\mathbb R)}^2
+Q\sum_{j=1}^2\|\nabla P_\Lambda^\perp y_j\|_{L^2(\Omega;\mathbb R^d)}^2\\
&-\sum_{j=1}^2\langle f_j,\Delta P_\Lambda^\perp y_j\rangle_{L^2(\Omega;\mathbb R)}
+\mu_\Lambda\mathcal X_E(t)\sum_{j=1}^2\langle\mathcal X_{\Dd_t}P_\Lambda y_j,\Delta P_\Lambda^\perp y_j\rangle_{L^2(\Omega;\mathbb R)}.
\end{aligned}
\end{equation*}
Combining these two identities, we obtain
\begin{equation}\label{eq:energy-identity}
\begin{aligned}
\frac{d}{dt}V_\Lambda(y)\le{}&2QV_\Lambda(y)-2\gamma_\Lambda\|\nabla P_\Lambda y\|_{L^2(\Omega;\mathbb C^d)}^2-2\|\Delta P_\Lambda^\perp y\|_{L^2(\Omega;\mathbb C)}^2\\
&-2\mu_\Lambda\gamma_\Lambda\mathcal X_E(t)\|\mathcal X_{\Dd_t}P_\Lambda y\|_{L^2(\Omega;\mathbb C)}^2
+2\gamma_\Lambda\sum_{j=1}^2\langle f_j,P_\Lambda y_j\rangle_{L^2(\Omega;\mathbb R)}\\
&-2\sum_{j=1}^2\langle f_j,\Delta P_\Lambda^\perp y_j\rangle_{L^2(\Omega;\mathbb R)}
+2\mu_\Lambda\mathcal X_E(t)\sum_{j=1}^2\langle\mathcal X_{\Dd_t}P_\Lambda y_j,\Delta P_\Lambda^\perp y_j\rangle_{L^2(\Omega;\mathbb R)}.
\end{aligned}
\end{equation}

For almost every $t\in E$, Lemma~\ref{lem:fubini} and the spectral inequality \eqref{eq:spectral-measurable} give, for $j=1,2$,
\begin{equation*}
\mathcal X_E(t)\|\mathcal X_{\Dd_t}P_\Lambda y_j\|_{L^2(\Omega;\mathbb R)}^2\ge C_0^{-1}e^{-C_0\sqrt\Lambda}\mathcal X_E(t)\|P_\Lambda y_j\|_{L^2(\Omega;\mathbb R)}^2.
\end{equation*}
Using \eqref{eq:para-selec}, we obtain
$$2\mu_\Lambda\gamma_\Lambda\mathcal X_E(t)\|\mathcal X_{\Dd_t}P_\Lambda y\|_{L^2(\Omega;\mathbb C)}^2\ge2\Lambda\gamma_\Lambda\mathcal X_E(t)\|P_\Lambda y\|_{L^2(\Omega;\mathbb C)}^2.$$
On the other hand, Young's inequality and $\mu_\Lambda^2=\Lambda\gamma_\Lambda$ imply
$$2\mu_\Lambda\mathcal X_E(t)\sum_{j=1}^2\bigl|\langle\mathcal X_{\Dd_t}P_\Lambda y_j,\Delta P_\Lambda^\perp y_j\rangle_{L^2(\Omega;\mathbb R)}\bigr|\le\Lambda\gamma_\Lambda\mathcal X_E(t)\|P_\Lambda y\|_{L^2(\Omega;\mathbb C)}^2+\mathcal X_E(t)\|\Delta P_\Lambda^\perp y\|_{L^2(\Omega;\mathbb C)}^2.$$
Using the spectral gap, we have
$$\|\Delta P_\Lambda^\perp y\|_{L^2(\Omega;\mathbb C)}^2\ge\Lambda\|\nabla P_\Lambda^\perp y\|_{L^2(\Omega;\mathbb C^d)}^2.$$
Combining these estimates with \eqref{eq:energy-identity}, we arrive at
\begin{equation}\label{eq:no-f}
\begin{aligned}
\frac{d}{dt}V_\Lambda(y)\le{}&\left(2Q-\frac\Lambda2\mathcal X_E(t)\right)V_\Lambda(y)-2\gamma_\Lambda\|\nabla P_\Lambda y\|_{L^2(\Omega;\mathbb C^d)}^2-\frac12\|\Delta P_\Lambda^\perp y\|_{L^2(\Omega;\mathbb C)}^2\\
&+2\gamma_\Lambda\sum_{j=1}^2\|f_j\|_{L^2(\Omega;\mathbb R)}\|P_\Lambda y_j\|_{L^2(\Omega;\mathbb R)}
+2\sum_{j=1}^2\|f_j\|_{L^2(\Omega;\mathbb R)}\|\Delta P_\Lambda^\perp y_j\|_{L^2(\Omega;\mathbb R)}.
\end{aligned}
\end{equation}

We now estimate the cubic terms. By \eqref{eq:sobolev-H1-L6}, for $j=1,2$,
$$\|f_j\|_{L^2(\Omega;\mathbb R)}\le\sqrt{1+\beta^2}\,\|y\|_{L^6(\Omega;\mathbb C)}^3\le c_0\|y\|_{H_0^1(\Omega;\mathbb C)}^3.$$
Applying Poincar\'e's and Young's inequalities to the last line of \eqref{eq:no-f}, we obtain, for a generic constant $C_*>0$,
\begin{equation}\label{eq:complex-cubic-absorb}
\begin{aligned}
&2\sum_{j=1}^2\gamma_\Lambda\|f_j\|_{L^2(\Omega;\mathbb R)}\|P_\Lambda y_j\|_{L^2(\Omega;\mathbb R)}+2\sum_{j=1}^2\|f_j\|_{L^2(\Omega;\mathbb R)}\|\Delta P_\Lambda^\perp y_j\|_{L^2(\Omega;\mathbb R)}\\
&\quad\quad\qquad\le\gamma_\Lambda\|\nabla P_\Lambda y\|_{L^2(\Omega;\mathbb C^d)}^2+\frac14\|\Delta P_\Lambda^\perp y\|_{L^2(\Omega;\mathbb C)}^2
+C_*(\gamma_\Lambda+1)\|y\|_{H_0^1(\Omega;\mathbb C)}^6.
\end{aligned}
\end{equation}
Moreover, by Poincar\'e's inequality and the spectral gap, there exists $c_*>0$ depending only on $\Omega$ such that
$$\gamma_\Lambda\|\nabla P_\Lambda y\|_{L^2(\Omega;\mathbb C^d)}^2+\frac14\|\Delta P_\Lambda^\perp y\|_{L^2(\Omega;\mathbb C)}^2\ge c_*\|y\|_{H_0^1(\Omega;\mathbb C)}^2.$$
It follows from \eqref{eq:no-f} and \eqref{eq:complex-cubic-absorb} that
\begin{equation}\label{eq:has-f}
\frac{d}{dt}V_\Lambda(y)\le\left(2Q-\frac\Lambda2\mathcal X_E(t)\right)V_\Lambda(y)+\|y\|_{H_0^1(\Omega;\mathbb C)}^2\bigl(C_*(\gamma_\Lambda+1)\|y\|_{H_0^1(\Omega;\mathbb C)}^4-c_*\bigr).
\end{equation}

Set
\begin{equation}
\sigma_\Lambda=\left(\frac{c_*}{2C_*(\gamma_\Lambda+1)}\right)^{1/4}.
\end{equation}
If $\|y(t)\|_{H_0^1(\Omega;\mathbb C)}\le\sigma_\Lambda$ on an initial part of the maximal existence interval, then the last term of \eqref{eq:has-f} is nonpositive, and therefore
\begin{equation}\label{eq:raw-stabilization}
V_\Lambda(y(t))\le\exp\left(2Q(t-t_n)-\frac\Lambda2|E\cap(t_n,t)|\right)V_\Lambda(y(t_n)).
\end{equation}
It follows from \eqref{eq:norm-equ} that
\begin{equation}\label{eq:complex-bootstrap}
\|y(t)\|_{H_0^1(\Omega;\mathbb C)}\le\sqrt{2\gamma_\Lambda}\,e^{|Q|T}\|y(t_n)\|_{H_0^1(\Omega;\mathbb C)}.
\end{equation}

Suppose that $y(t_n)$ satisfies
\begin{equation}\label{eq:rho-meas}
\|y(t_n)\|_{H_0^1(\Omega;\mathbb C)}\le\frac{\sigma_\Lambda}{2\sqrt{2\gamma_\Lambda}}e^{-|Q|T}.
\end{equation}
If $\|y(t)\|_{H_0^1(\Omega;\mathbb C)}$ reaches $\sigma_\Lambda$ before $t_{n+1}$, let $\tau$ be the first such time. Applying \eqref{eq:complex-bootstrap} on $[t_n,\tau]$ and using \eqref{eq:rho-meas}, we obtain
\begin{equation*}
\sigma_\Lambda=\|y(\tau)\|_{H_0^1(\Omega;\mathbb C)}
\le\sqrt{2\gamma_\Lambda}\,e^{|Q|T}\|y(t_n)\|_{H_0^1(\Omega;\mathbb C)}
\le\frac{\sigma_\Lambda}{2},
\end{equation*}
a contradiction. Hence the solution remains bounded by $\sigma_\Lambda$ on its maximal existence interval. The continuation criterion in Lemma~\ref{lem:we-pos-th} then extends the solution to $t_{n+1}$.

Since $\gamma_\Lambda$ and $\mu_\Lambda$ are given by \eqref{eq:para-selec}, all polynomial factors in $\Lambda$ can be absorbed into an exponential $e^{C\sqrt\Lambda}$ for $\Lambda>1$. We may therefore choose $C\ge1$, independent of $\Lambda$ and $n$, such that
\begin{equation}\label{eq:complex-parameter-bounds}
e^{-C\sqrt\Lambda}\le\frac{\sigma_\Lambda}{2\sqrt{2\gamma_\Lambda}},\qquad (1+\mu_\Lambda)\sqrt{2\gamma_\Lambda}\le e^{C\sqrt\Lambda}.
\end{equation}
The first estimate gives \eqref{eq:rho-meas}, and the second gives \eqref{eq:feedback-lambda-bound} and \eqref{eq:y-decay-meas-1}. Finally, applying \eqref{eq:norm-equ} to \eqref{eq:raw-stabilization} at $t=t_{n+1}$, we find
$$\|y(t_{n+1})\|_{H_0^1(\Omega;\mathbb C)}\le\sqrt{2\gamma_\Lambda}\,e^{|Q|T}e^{-\Lambda|E\cap I_n|/4}\|y(t_n)\|_{H_0^1(\Omega;\mathbb C)}.$$
Using $|E\cap I_n|\ge\ell_n/2$ from \eqref{eq:active-fraction} and \eqref{eq:complex-parameter-bounds} completes the proof.
\end{proof}

\begin{proposition}\label{prop:complex-oneinterval}
There exist constants $A_{\mathbb C},b_{\mathbb C},B_{\mathbb C},m_{\mathbb C}>0$, independent of $q,r_0,n$, such that with
$$\Lambda_n=\frac{A_{\mathbb C}^2}{\ell_n^2},\;\;\qquad \mathcal K_n^{\mathbb C}(t)=\mathcal X_E(t)\mathcal K_{\Lambda_n},\;\;t\in I_n,$$
the condition
\begin{equation}
\|y(t_n)\|_{H_0^1(\Omega;\mathbb C)}\le e^{-B_{\mathbb C}/\ell_n}
\end{equation}
ensures that equation \eqref{eq:CGL} with $u(t)=\mathcal K_n^{\mathbb C}(t)y(t)$ has a unique strong solution on $[t_n,t_{n+1}]$ satisfying
\begin{align}
\|y(t_{n+1})\|_{H_0^1(\Omega;\mathbb C)}&\le e^{-b_{\mathbb C}/\ell_n}\|y(t_n)\|_{H_0^1(\Omega;\mathbb C)},\label{eq:complex-interval-contract}\\
\|y(t)\|_{H_0^1(\Omega;\mathbb C)}+\|\mathcal X_{\Dd_t}\mathcal K_n^{\mathbb C}(t)y(t)\|_{L^2(\Omega;\mathbb C)}&\le e^{m_{\mathbb C}/\ell_n}\|y(t_n)\|_{H_0^1(\Omega;\mathbb C)},\quad\text{a.e. }t\in I_n.
\end{align}
Moreover, $\|\mathcal K_n^{\mathbb C}(t)\|_{\mathcal L(L^2(\Omega;\mathbb C))}\le e^{m_{\mathbb C}/\ell_n}$ after enlarging $m_{\mathbb C}$.
\end{proposition}

\begin{proof}
Let $C$ be as in Proposition~\ref{prop:ra-sta-new}. Set $a=|Q|+C$ and choose $A_{\mathbb C}>\max\{2,16a\}$. Then $\Lambda_n>1$, and
$$e^{|Q|T+C\sqrt{\Lambda_n}}\le\exp\!\left(\frac{aA_{\mathbb C}}{\ell_n}\right).$$
Set $B_{\mathbb C}=|Q|+CA_{\mathbb C}+1$, so that $e^{-B_{\mathbb C}/\ell_n}\le\rho_{\Lambda_n,T}$. The estimate throughout the interval and the bound for the feedback follow from \eqref{eq:y-decay-meas-1} and \eqref{eq:feedback-lambda-bound} after choosing $m_{\mathbb C}$ sufficiently large. By \eqref{eq:y-decay-meas-2},
$$\|y(t_{n+1})\|_{H_0^1(\Omega;\mathbb C)}\le\exp\!\left(\frac{aA_{\mathbb C}-A_{\mathbb C}^2/8}{\ell_n}\right)\|y(t_n)\|_{H_0^1(\Omega;\mathbb C)}.$$
Thus taking $b_{\mathbb C}=A_{\mathbb C}^2/8-aA_{\mathbb C}>0$ proves \eqref{eq:complex-interval-contract}.
\end{proof}

\section{Stabilization on one interval with a real control}\label{sec:prelim}
Throughout this section, we assume that $\alpha\ne0$. For $\Lambda\ge\lambda_1$, we use the spectral projection $P_\Lambda$ defined in \eqref{eq:low-fre-pro} and set
\begin{equation}
 \Pi_\Lambda=\operatorname{diag}(P_\Lambda,P_\Lambda).
\end{equation}
For $c\in\R^{N(\Lambda)}$, set
$$\Phi_\Lambda c=\sum_{j=1}^{N(\Lambda)}c_j\varphi_j, \qquad G_\Lambda(\om)=\bigl(\langle\mathcal X_\om\varphi_j,\varphi_k\rangle_{L^2(\Omega;\mathbb R)}\bigr)_{j,k\le N(\Lambda)}.$$

Then
$$c^\top G_\Lambda(\om)c=\|\Phi_\Lambda c\|_{L^2(\om)}^2\;\;\; \text{and}\;\;\; 0\le G_\Lambda(\om)\le I.$$
Moreover, we have
$$\|\mathcal X_\om\Phi_\Lambda U\|_{L^2(\Omega;\mathbb R)}\le\|U\|_{\R^{N(\Lambda)}},\qquad \|(I-P_\Lambda)(\mathcal X_\om\Phi_\Lambda U)\|_{L^2(\Omega;\mathbb R)}\le\|U\|_{\R^{N(\Lambda)}}.$$
By \eqref{eq:spectral-measurable}, for almost every $t\in E$,
\begin{equation}\label{eq:G-coercive}
 s_\Lambda I\le G_\Lambda(\Dd_t)\le I,
 \qquad s_\Lambda=C_0^{-1}e^{-C_0\sqrt\Lambda}.
\end{equation}

Fix an interval $I_n=(t_n,t_{n+1})$ from Lemma~\ref{lem:measurable-geometry}. We introduce the cutoff $\Lambda_n$ on $I_n$ and set $N_n=N(\Lambda_n)$. We consider finite-dimensional real controls of the form
\begin{equation}
 u_{n}(t)=\mathcal X_E(t)\Phi_{\Lambda_n} U_{n}(t),
 \qquad U_{n}(t)\in\R^{N_n}.
\end{equation}

If
$$
 \Pi_{\Lambda_n} X=\left(\sum_{k=1}^{N_n}p_k\varphi_k,\sum_{k=1}^{N_n}q_k\varphi_k\right),$$
we define the low-frequency vector
$Z_{n}=(p_1,\ldots,p_{N_n},q_1,\ldots,q_{N_n})^\top.$
Let $F_{n}(X)\in\R^{2N_{n}}$ denote the coefficient vector of $\Pi_{\Lambda_n}\Ff(X)$. Projecting the controlled equation onto the first $N_n$ modes gives
\begin{equation}\label{eq:low-frequency-system}
 \dot Z_{n}(t)
 =A_{n} Z_{n}(t)+B_{n}(t)U_{n}(t)+F_n(X(t)),
 \quad\text{for a.e. }t\in I_n,
\end{equation}
where
$$
 A_n=
 \begin{pmatrix}QI-D_n&\alpha D_n\\-\alpha D_n&QI-D_n\end{pmatrix},
 \qquad D_n=\operatorname{diag}(\lambda_1,\ldots,\lambda_{N_n}),
$$
and
$$
 B_{n}(t)=\mathcal X_E(t)b_{n}(\Dd_t),
 \qquad
 b_n(\om)=\binom{G_{\Lambda_n}(\om)}{0}.
$$
The complementary high frequencies satisfy
\begin{equation}\label{eq:high-gap}
 \|\Aop(I-\Pi_{\Lambda_n})X\|_{\Hh}^2
 \ge\Lambda_n\|\Aop^{1/2}(I-\Pi_{\Lambda_n})X\|_{\Hh}^2.
\end{equation}

We first establish rapid stabilization for the linear part of the system of low frequencies \eqref{eq:low-frequency-system}, and derive the stabilization estimate on each $I_n$.

\subsection{Stabilization of the low frequencies by the Gramian method}

Fix $I_n$ and choose a decay rate $\sigma_n>0$. To obtain decay at the rate $\sigma_n$, we shift the generator of the low frequencies by $\sigma_n I$ and exponentially stabilize the resulting system. We use the following notation:
\begin{equation}\label{eq:Lambda-sigma}
 \Lambda_n=Q+\frac{15}{16}\sigma_n, \qquad N_n=N(\Lambda_n), \qquad \mathcal A_n=A_n+\sigma_n I_{2N_n}.
\end{equation}
We take $\sigma_n\ge\sigma_0$, where $\sigma_0\ge8$ is large enough that $\Lambda_n\ge\lambda_1$ and $\sigma_n\ge16|Q|$.

We extend the input matrix periodically and apply the infinite-horizon Gramian construction; see also \cite{DeNicolaoStrada1997} for the stabilization of linear periodic systems by the Gramian method. In the local time variable $s=t-t_n\in(0,\ell_n)$, set $\widetilde B_{n}(s)=B_{n}(t_n+s)$ and extend this matrix $\ell_n$-periodically to $\R$. We assign it the value zero at exceptional times. By \eqref{eq:active-fraction}, we have, for every $s\in\R$,
\begin{equation}\label{eq:period-active}
	\bigl|\{r\in(0,\ell_n):\widetilde B_{n}(s+r)\ne0\}\bigr|
	\ge\frac12\ell_n.
\end{equation}

For the linear part of \eqref{eq:low-frequency-system}, consider $0\le s\le\ell_n$ and make the change of variables
$$Z(s)=e^{\sigma_ns}Z_{n}(t_n+s),\qquad U(s)=e^{\sigma_ns}U_{n}(t_n+s).$$
The resulting system on $[0,\ell_n]$ is
\begin{equation}\label{eq:shifted-system}
\dot Z=\mathcal A_nZ+\widetilde B_n(s)U,\qquad Z(0)=Z_{n}(t_n).
\end{equation}
For $\lambda_k\le\Lambda_n$, the eigenvalues of the $k$th block of $\mathcal A_n$ are $ Q-\lambda_k+\sigma_n\pm i\alpha\lambda_k. $ Since $\lambda_k\le\Lambda_n$ and $\sigma_n\ge16|Q|$, we have $\operatorname{Re}(Q-\lambda_k+\sigma_n\pm i\alpha\lambda_k)\ge\frac{\sigma_n}{16}$. Each block is normal, and therefore
\begin{equation}\label{eq:shift-semigroup}
 \norm{e^{-r\mathcal A_n}}_{\mathcal L(\R^{2N_n})}\le e^{-\sigma_n r/16},\qquad\forall r\ge0.
\end{equation}

For the exponential stabilization of \eqref{eq:shifted-system}, we introduce the periodic Gramian
\begin{equation}\label{eq:Wn}
 W_n(s)=\int_0^\infty e^{-r\mathcal A_n} \widetilde B_n(s+r)\widetilde B_n(s+r)^\top e^{-r\mathcal A_n^\top}\dd r.
\end{equation}
The following proposition uses this Gramian to construct a feedback which stabilizes \eqref{eq:shifted-system}.
\begin{proposition}\label{prop:Riccati}
The matrix $W_n(s)$ is positive definite for every $s\in\R$. Moreover, $W_n(\cdot)$ and $R_n(\cdot)=W_n^{-1}(\cdot)$ are $\ell_n$-periodic and locally Lipschitz on $\R$. The matrix $R_n(\cdot)$ satisfies the Riccati equation
\begin{equation}\label{eq:real-riccati-main}
\dot R_n(s)+\mathcal A_n^\top R_n(s)+R_n(s)\mathcal A_n-R_n(s)\widetilde B_n(s)\widetilde B_n(s)^\top R_n(s)=0,\qquad\text{a.e. }s\in\mathbb R.
\end{equation}
For any $s_0\in\R$ and $Z_0\in\R^{2N_n}$, consider the periodically extended system \eqref{eq:shifted-system} on $[s_0,\infty)$ with initial condition $Z(s_0)=Z_0$. The minimum of $\int_{s_0}^\infty\|U(s)\|_{\R^{N_n}}^2\dd s$ among $U\in L^2(s_0,\infty;\R^{N_n})$ for which $\lim\limits_{s\to \infty}\|Z(s)\|_{\R^{2N_n}}=0$ equals $ Z_0^\top R_n(s_0)Z_0, $ and the optimal control is generated by the feedback
\begin{equation}\label{eq:LQ-feedback}
 U(s)=-\widetilde B_n(s)^\top R_n(s)Z(s),\qquad s\ge s_0.
\end{equation}
This feedback exponentially stabilizes the shifted system on $[s_0,\infty)$.
\end{proposition}

\begin{proof}
By \eqref{eq:shift-semigroup} and the bound $\|\widetilde B_n(s)\|_{\mathcal L(\R^{N_n},\R^{2N_n})}\le1$, we have $0\le W_n(s)\le8\sigma_n^{-1}I$. 
Fix $s\in\R$ and let $z=(a_1,\ldots,a_{N_n},b_1,\ldots,b_{N_n})^\top\in\R^{2N_n}$. After reordering the coordinates, the $k$th block of $\mathcal A_n^\top$ is
$$\begin{pmatrix}c_{k,n}&-\alpha\lambda_k\\ \alpha\lambda_k&c_{k,n}\end{pmatrix},\qquad c_{k,n}=Q-\lambda_k+\sigma_n.$$
Therefore the first component of the $k$th block of $e^{-r\mathcal A_n^\top}z$ is
\begin{equation}\label{eq:adjoint-first-component}
 e^{-c_{k,n}r}\bigl(a_k\cos(\alpha\lambda_k r)+b_k\sin(\alpha\lambda_k r)\bigr).
\end{equation}
Let $v(r)\in\R^{N_n}$ have the components \eqref{eq:adjoint-first-component}. By periodicity, the upper block of $\widetilde B_n(s+r)$ satisfies \eqref{eq:G-coercive} for almost every $r>0$ such that $\widetilde B_n(s+r)\ne0$. Hence for these $r$,
\begin{equation}\label{eq:G-square-lower}
 \bigl\|\widetilde B_n(s+r)^\top e^{-r\mathcal A_n^\top}z\bigr\|_{\R^{N_n}}^2\ge C_0^{-2}e^{-2C_0\sqrt{\Lambda_n}}\|v(r)\|_{\R^{N_n}}^2.
\end{equation}

Suppose that $z^\top W_n(s)z=0$, and set
$$E_{n,s}=\{r\in(0,\ell_n):\widetilde B_n(s+r)\ne0\}.$$
By \eqref{eq:period-active}, $|E_{n,s}|\ge\ell_n/2>0$. Using \eqref{eq:Wn}, \eqref{eq:adjoint-first-component} and \eqref{eq:G-square-lower}, we have
$$\begin{aligned}
0=z^\top W_n(s)z&=\int_0^\infty\bigl\|\widetilde B_n(s+r)^\top e^{-r\mathcal A_n^\top}z\bigr\|_{\R^{N_n}}^2\dd r\\
&\ge C_0^{-2}e^{-2C_0\sqrt{\Lambda_n}}\sum_{k=1}^{N_n}\int_{E_{n,s}}e^{-2c_{k,n}r}|a_k\cos(\alpha\lambda_k r)+b_k\sin(\alpha\lambda_k r)|^2\dd r\ge0.
\end{aligned}$$
Each integrand is nonnegative and $e^{-2c_{k,n}r}>0$. It follows that, for every $k=1,\ldots,N_n$,
$$a_k\cos(\alpha\lambda_k r)+b_k\sin(\alpha\lambda_k r)=0,\quad\text{for a.e. }r\in E_{n,s}.$$
Since these functions are analytic and vanish on a set of positive measure, they vanish identically. Evaluating the functions and their derivatives at $r=0$ gives $a_k=0$ and $\alpha\lambda_k b_k=0$. Since $\alpha\ne0$ and $\lambda_k>0$, we have $z=0$. Thus $W_n(s)$ is positive definite.

Set $C_n=\widetilde B_n\widetilde B_n^\top$. Splitting \eqref{eq:Wn} at $h>0$ gives
$$W_n(s)=\int_0^h e^{-r\mathcal A_n}C_n(s+r)e^{-r\mathcal A_n^\top}\dd r+e^{-h\mathcal A_n}W_n(s+h)e^{-h\mathcal A_n^\top}.$$
Since $C_n$ and $W_n$ are bounded, this identity implies that $W_n$ is locally Lipschitz and
$$\dot W_n(s)=\mathcal A_nW_n(s)+W_n(s)\mathcal A_n^\top-C_n(s),\;\;\;\text{for a.e. }s\in\R.$$
Since $W_n$ is continuous, positive definite and periodic, $R_n$ is bounded and locally Lipschitz. The identity $\dot R_n=-R_n\dot W_nR_n$ then gives \eqref{eq:real-riccati-main}.

For the feedback \eqref{eq:LQ-feedback}, the corresponding state and control are given, for $s\ge s_0$, by
$$\begin{aligned}
Z_*(s)&=W_n(s)e^{-(s-s_0)\mathcal A_n^\top}R_n(s_0)Z_0,\quad U_*(s)=-\widetilde B_n(s)^\top e^{-(s-s_0)\mathcal A_n^\top}R_n(s_0)Z_0.
\end{aligned}$$
Indeed, $Z_*(s_0)=Z_0$, and the identity for $\dot W_n$ gives $\dot Z_* =\mathcal A_nZ_*+\widetilde B_nU_*$. The boundedness of $W_n$ and $\widetilde B_n$, together with \eqref{eq:shift-semigroup}, yields exponential decay of $Z_*$ and $U_*$. Thus this pair is admissible.

For any admissible pair $(Z,U)$, \eqref{eq:real-riccati-main} gives
$$\frac{d}{ds}(Z^\top R_nZ)+\|U\|_{\R^{N_n}}^2=\|U+\widetilde B_n^\top R_nZ\|_{\R^{N_n}}^2.$$
Since $R_n$ is bounded on $\mathbb R$ and $Z(s)\to0$ as $s\to+\infty$, we derive $\int_{s_0}^\infty\|U\|_{\R^{N_n}}^2\dd s\ge Z_0^\top R_n(s_0)Z_0$, with equality for $(Z_*,U_*)$. This completes the proof.
\end{proof}

The feedback for the original linear low-frequency system on $I_n$ is
$$U_{n}(t)=-B_{n}(t)^\top R_n(t-t_n)Z_{n}(t),\qquad t\in I_n.$$
Set $M_n=\sup_{s\in\R}\|R_n(s)\|_{\mathcal L(\R^{2N_n})}$. Taking $s_0=0$ and $Z_0=Z_{n}(t_n)$, we return to the original variables through
$$Z_{n}(t_n+s)=e^{-\sigma_n s} Z_*(s),\quad U_{n}(t_n+s)=e^{-\sigma_n s} U_*(s).$$
Using $W_n(s)\le8/\sigma_n I$, $R_n\le M_n I$, $\|\widetilde B_n(s)\|_{\mathcal L(\R^{N_n},\R^{2N_n})}\le1$, and \eqref{eq:shift-semigroup}, we find, for almost every $s\in(0,\ell_n)$,
$$\begin{aligned}
\|Z_{n}(t_n+s)\|_{\R^{2N_n}}\le\frac{8M_n}{\sigma_n}e^{-\sigma_ns}\|Z_{n}(t_n)\|_{\R^{2N_n}},\quad
\|U_{n}(t_n+s)\|_{\R^{N_n}}\le M_ne^{-\sigma_ns}\|Z_{n}(t_n)\|_{\R^{2N_n}}.
\end{aligned}$$
This proves rapid stabilization for the low frequencies on $I_n$. To estimate the stabilization cost, it remains to bound $M_n$.
\subsection{Quantitative estimates for the low frequencies}
By \eqref{eq:adjoint-first-component}, after removing the scalar exponential factor, the observed component has the form
$$a\cos(\rho r)+b\sin(\rho r),\qquad \rho=\alpha\lambda_k.$$
We first estimate this signal on arbitrary measurable subsets occupying a fixed positive fraction of the observation interval.

\begin{lemma}\label{lem:measurable-rotation}
For every $\theta\in(0,1)$ there is $c_\theta>0$ such that, if $F\subset(0,\tau)$ is measurable, $|F|\ge\theta\tau$, $0<\tau\le1$, and $\rho\ne0$, then for all $a,b\in\mathbb R$,
\begin{equation}\label{eq:measurable-rotation}
 \int_F|a\cos(\rho r)+b\sin(\rho r)|^2\dd r \ge c_\theta\tau\min\{1,\rho^2\tau^2\}(a^2+b^2).
\end{equation}
Consequently, if $\alpha\ne0$, then for every $k\ge1$,
\begin{equation}\label{eq:modal-measurable-rotation}
 \int_F|a\cos(\alpha\lambda_k r)+b\sin(\alpha\lambda_k r)|^2\dd r \ge c_{\theta,\alpha,\lambda_1}\tau^3(a^2+b^2),
\end{equation}
where $c_{\theta,\alpha,\lambda_1}=c_\theta\min\{1,\alpha^2\lambda_1^2\}>0$.
\end{lemma}

\begin{proof}
The case $a=b=0$ is trivial. Otherwise let $A_0=(a^2+b^2)^{1/2}$ and choose $\phi$ so that
$$a\cos(\rho r)+b\sin(\rho r)=A_0\cos(|\rho|r-\phi).$$
Set $L=|\rho|\tau$ and $H=|\rho|F\subset(0,L)$. Then $|H|\ge\theta L$ and
\begin{equation}\label{eq:rotation-rescale}
 \int_F|a\cos(\rho r)+b\sin(\rho r)|^2\dd r =\frac{A_0^2}{|\rho|}\int_H\cos^2(x-\phi)\dd x.
\end{equation}

There is a constant $C_{\mathrm{rot}}\ge1$ such that, for every interval $J$ and $0<\eta<1$,
\begin{equation}\label{eq:cos-sublevel}
 \bigl|\{x\in J:|\cos(x-\phi)|\le\eta\}\bigr| \le C_{\mathrm{rot}}\eta(|J|+1).
\end{equation}
Indeed, $|\cos|$ has period $\pi$; in one period the sublevel set has measure $2\arcsin\eta\le\pi\eta$, and an interval $J$ meets at most $|J|/\pi+2$ period cells. Apply \eqref{eq:cos-sublevel} to $J=(0,L)$ with $\eta=\frac{\theta L}{2C_{\mathrm{rot}}(L+1)}.$ The sublevel set then has measure at most $\theta L/2$, so $H_\eta=H\cap\{|\cos(x-\phi)|>\eta\}$ has measure at least $\theta L/2$. Therefore
$$\int_H\cos^2(x-\phi)\dd x \ge\eta^2|H_\eta| \ge c_\theta\frac{L^3}{(L+1)^2} \ge c_\theta L\min\{1,L^2\}.$$
Together with \eqref{eq:rotation-rescale}, this proves \eqref{eq:measurable-rotation}. For $\rho=\alpha\lambda_k$, since $\lambda_k\ge\lambda_1$ and $0<\tau\le1$,
$$\min\{1,\alpha^2\lambda_k^2\tau^2\} \ge \tau^2\min\{1,\alpha^2\lambda_1^2\},$$
and \eqref{eq:modal-measurable-rotation} follows.
\end{proof}

To estimate the inverse Gramian, we use a quantitative observability inequality as in \cite{NguyenSch,NguyenKdV}. Here we combine Lemma~\ref{lem:measurable-rotation} with \eqref{eq:G-coercive} to estimate $W_n(s)$ uniformly in $s$.
\begin{proposition}
The Gramian $W_n$ satisfies the following bounds uniformly in $s$:
\begin{equation}\label{eq:Wn-bounds}
 cC_0^{-2}\ell_n^3 e^{-2C_0\sqrt{\Lambda_n}-\frac32\sigma_n\ell_n-C|Q|\ell_n}I \le W_n(s)\le\frac8{\sigma_n}I.
\end{equation}
The constants $c,C>0$ are independent of $n$, $\ell_n$, $\sigma_n$, and $\Lambda_n$.
\end{proposition}

\begin{proof}
The upper bound follows from \eqref{eq:shift-semigroup} and $\|\widetilde B_n(s)\|_{\mathcal L(\R^{N_n},\R^{2N_n})}\le1$. We now prove the lower bound.

By \eqref{eq:period-active}, the active times in $(0,\ell_n)$ have measure at least $\ell_n/2$. Hence
$$F_{n,s}:=\{r\in(0,3\ell_n/4):\widetilde B_n(s+r)\ne0\}$$
satisfies $|F_{n,s}|\ge\ell_n/4$, because the remaining interval $(3\ell_n/4,\ell_n)$ has length $\ell_n/4$. Consequently,
$$|F_{n,s}|\ge \frac 1 3|(0,3\ell_n/4)|.$$

Let $z=(a_1,\ldots,a_{N_n},b_1,\ldots,b_{N_n})^\top$ and use the notation of \eqref{eq:adjoint-first-component}.
Since $c_{k,n}\le\sigma_n+|Q|$, we have, for $0<r<3\ell_n/4$,
\begin{equation}\label{eq:exp-weight-lower}
 e^{-2c_{k,n}r} \ge e^{-2(\sigma_n+|Q|)r} \ge e^{-\frac32(\sigma_n+|Q|)\ell_n}.
\end{equation}
Combining \eqref{eq:adjoint-first-component}, \eqref{eq:G-square-lower} and \eqref{eq:exp-weight-lower}, we find, for almost every $r\in F_{n,s}$,
\begin{align}
 \bigl\|\widetilde B_n(s+r)^\top e^{-r\mathcal A_n^\top}z\bigr\|_{\R^{N_n}}^2 &\ge C_0^{-2}e^{-2C_0\sqrt{\Lambda_n}} e^{-\frac32(\sigma_n+|Q|)\ell_n}\times\sum_{k=1}^{N_n} |a_k\cos(\alpha\lambda_k r)+b_k\sin(\alpha\lambda_k r)|^2. \label{eq:observed-lower}
\end{align}

Set $\tau_n=3\ell_n/4$. Since $\ell_n\le1$ and $|F_{n,s}|\ge\tau_n/3$, Lemma~\ref{lem:measurable-rotation} gives, for every $k$,
$$\int_{F_{n,s}} |a_k\cos(\alpha\lambda_k r)+b_k\sin(\alpha\lambda_k r)|^2\dd r \ge c\tau_n^3(a_k^2+b_k^2) \ge c\ell_n^3(a_k^2+b_k^2).$$
Integrating \eqref{eq:observed-lower} over $F_{n,s}$, we obtain the lower bound in \eqref{eq:Wn-bounds}.
\end{proof}

Consequently, $M_n\le\mathfrak R_n$, where
\begin{equation}\label{eq:Rn-bounds}
 \frac{\sigma_n}{8}I\le R_n(s)\le\mathfrak R_n I, \qquad \mathfrak R_n=CC_0^2\ell_n^{-3} \exp\!\left(2C_0\sqrt{\Lambda_n}+\frac32\sigma_n\ell_n+C|Q|\ell_n\right),
\end{equation}
for a generic constant $C\ge 1$ independent of $n$, $\ell_n$, $\sigma_n$, and $\Lambda_n$.
\subsection{A Lyapunov estimate on one interval}

We now apply the frequency--Lyapunov method to the full nonlinear system by combining the quantitative stabilization estimate for the low frequencies with the dissipation of the high frequencies.
\vspace{1mm}

For $t\in I_n$, put $s=t-t_n$. We define the feedback control
\begin{equation}\label{eq:physical-feedback}
 U_n(t,X)=-B_n(t)^\top R_n(t-t_n)Z_n(X), \qquad \mathcal K_n^{\mathbb R}(t)X=\Phi_{\Lambda_n}U_n(t,X).
\end{equation}
Here $U_n\in\R^{N_n}$ is the coefficient vector in the finite-dimensional control space introduced in Section~\ref{sec:prelim}. The feedback acts on $y=p+iq$ through $X=(p,q)^\top$. For the high-frequency component, we have
\begin{equation}\label{eq:spillover}
 \mathcal B_{n,H}(t)U =(I-\Pi_{\Lambda_n})\binom{\mathcal X_{\Dd_t}\Phi_{\Lambda_n}U}{0}, \qquad \norm{\mathcal B_{n,H}(t)U}_{\Hh}\le \|U\|_{\R^{N_n}}.
\end{equation}
For fixed $n$, the local well-posedness on $I_n$ follows from Lemma~\ref{lem:wellposedness}.

Write
$$X_{L,n}=\Pi_{\Lambda_n}X, \qquad X_{H,n}=(I-\Pi_{\Lambda_n})X,$$
and define, for $t\in I_n$, the Lyapunov function
\begin{equation}\label{eq:Vn}
 \mathcal V_n(t,X)=Z_n^\top R_n(t-t_n)Z_n +\vartheta\norm{\Aop^{1/2}X_{H,n}}_{\Hh}^2, \qquad \vartheta=\frac1{256}.
\end{equation}

\begin{lemma}
The Lyapunov function $\mathcal V_n$ satisfies the following estimate with constants $c_V,C_V>0$ independent of $n$:
\begin{equation}\label{eq:Vn-equivalence}
 c_V\norm{X}_{\V}^2 \le\mathcal V_n(t,X) \le C_V\mathfrak R_n\norm{X}_{\V}^2, \quad \forall t\in I_n.
\end{equation}
\end{lemma}

\begin{proof}
The upper bound holds with $C_V=1$, since $R_n\le\mathfrak R_nI$, $\vartheta\le1$ and $\mathfrak R_n\ge1$. For the lower bound, \eqref{eq:Rn-bounds} and $\lambda_k\le\Lambda_n$ for $k\le N(\Lambda_n)$ give
$$\norm{\Aop^{1/2}X_{L,n}}_{\Hh}^2 \le\Lambda_n\|Z_n\|_{\R^{2N_n}}^2 \le\frac{8\Lambda_n}{\sigma_n}Z_n^\top R_nZ_n \le16Z_n^\top R_nZ_n.$$
Using Poincar\'e's inequality (see \cite{Brezis}), we find
$$\|X\|_{\V}^2\le(1+\lambda_1^{-1})\|\Aop^{1/2}X\|_{\Hh}^2\le256(1+\lambda_1^{-1})\mathcal V_n(t,X).$$
Thus the lower bound holds with $c_V=[256(1+\lambda_1^{-1})]^{-1}$.
\end{proof}

Set $F_{L,n}=F_n(X)$ and $F_{H,n}=(I-\Pi_{\Lambda_n})\Ff(X)$. Then \eqref{eq:low-frequency-system} and the complementary projection give
$$\dot Z_n=A_nZ_n+B_n(t)U_n+F_{L,n}, \qquad \partial_t X_{H,n}=\Ll X_{H,n}+\mathcal B_{n,H}(t)U_n+F_{H,n}.$$
By Lemma~\ref{lem:wellposedness} and the local Lipschitz continuity of $R_n$, the energy identities below hold for almost every time. Proposition~\ref{prop:Riccati} gives
\begin{equation}\label{eq:low-derivative}
 \frac{d}{dt}(Z_n^\top R_nZ_n) =-2\sigma_nZ_n^\top R_nZ_n-\|U_n\|_{\R^{N_n}}^2 +2\langle R_nZ_n,F_{L,n}\rangle_{\R^{2N_n}}.
\end{equation}
For the high frequencies, \eqref{eq:linear-energy} gives
\begin{align*}
\frac{d}{dt}\bigl(\vartheta\norm{\Aop^{1/2}X_{H,n}}_{\Hh}^2\bigr) &=-2\vartheta\norm{\Aop X_{H,n}}_{\Hh}^2 +2Q\vartheta\norm{\Aop^{1/2}X_{H,n}}_{\Hh}^2+2\vartheta\langle\Aop X_{H,n},\mathcal B_{n,H}U_n+F_{H,n}\rangle_{\Hh}.
\end{align*}

Using \eqref{eq:spillover} and Young's inequality, we obtain
$$\begin{aligned}
2\vartheta|\langle\Aop X_{H,n},\mathcal B_{n,H}U_n\rangle_{\Hh}|
 &\le\frac{\vartheta}{32}\norm{\Aop X_{H,n}}_{\Hh}^2+\frac18\|U_n\|_{\R^{N_n}}^2,\\
 2\vartheta|\langle\Aop X_{H,n},F_{H,n}\rangle_{\Hh}|
 &\le\frac{\vartheta}{32}\norm{\Aop X_{H,n}}_{\Hh}^2+\frac18\norm{\Ff(X)}_{\Hh}^2.
\end{aligned}$$
Together with \eqref{eq:high-gap} and \eqref{eq:Lambda-sigma}, these estimates imply
$$\begin{aligned}
&-\frac{31}{16}\vartheta\norm{\Aop X_{H,n}}_{\Hh}^2
 +2Q\vartheta\norm{\Aop^{1/2}X_{H,n}}_{\Hh}^2\le\vartheta\left(\frac{Q}{16}-\frac{465}{256}\sigma_n\right)
 \norm{\Aop^{1/2}X_{H,n}}_{\Hh}^2.
\end{aligned}$$
Since $\sigma_n\ge16|Q|$, the last coefficient is at most $-7\sigma_n/4$; hence
\begin{equation}\label{eq:high-linear-decay}
 -\frac{31}{16}\vartheta\norm{\Aop X_{H,n}}_{\Hh}^2 +2Q\vartheta\norm{\Aop^{1/2}X_{H,n}}_{\Hh}^2 \le-\frac74\sigma_n\vartheta\norm{\Aop^{1/2}X_{H,n}}_{\Hh}^2.
\end{equation}
Combining \eqref{eq:low-derivative}--\eqref{eq:high-linear-decay}, we arrive at
$$\frac{d}{dt}\mathcal V_n \le-\frac74\sigma_n\mathcal V_n-\frac78\|U_n\|_{\R^{N_n}}^2 +2|\langle R_nZ_n,F_{L,n}\rangle_{\R^{2N_n}}|+\frac18\norm{\Ff(X)}_{\Hh}^2.$$

Since $0<R_n\le\mathfrak R_nI$, we have $R_n^2\le\mathfrak R_nR_n$ and hence $\|R_nZ_n\|_{\R^{2N_n}}^2\le\mathfrak R_nZ_n^\top R_nZ_n\le\mathfrak R_n\mathcal V_n.$ Moreover, \eqref{eq:F-cubic} and \eqref{eq:Vn-equivalence} give $\norm{\Ff(X)}_{\Hh}\le C\mathcal V_n^{3/2}.$ We therefore obtain
\begin{equation}\label{eq:V-differential}
 \frac{d}{dt}\mathcal V_n \le-\frac74\sigma_n\mathcal V_n +C\mathfrak R_n^{1/2}\mathcal V_n^2 +C\mathcal V_n^3.
\end{equation}

Define
\begin{equation}\label{eq:eta-n}
 \eta_n=\min\left\{1, \frac{\sigma_n}{16C\mathfrak R_n^{1/2}}, \left(\frac{\sigma_n}{16C}\right)^{1/2}\right\}.
\end{equation}
As long as $\mathcal V_n\le\eta_n$, the last two terms in \eqref{eq:V-differential} are each bounded by $\sigma_n\mathcal V_n/16$, hence
\begin{equation}\label{eq:V-decay}
 \frac{d}{dt}\mathcal V_n\le-\frac{13}{8}\sigma_n\mathcal V_n.
\end{equation}
Assume
\begin{equation}\label{eq:bootstrap-initial}
 C_V\mathfrak R_n\norm{X(t_n)}_{\V}^2\le\frac12\eta_n.
\end{equation}
The same continuity argument as in the proof of Proposition~\ref{prop:ra-sta-new}, together with Lemma~\ref{lem:wellposedness}, extends the solution to $[t_n,t_{n+1}]$ and gives
$$\mathcal V_n(t)\le e^{-\frac{13}{8}\sigma_n(t-t_n)}\mathcal V_n(t_n)\le\frac12\eta_n,\quad t\in[t_n,t_{n+1}].$$
Consequently, \eqref{eq:V-decay} holds, and for almost every $t\in I_n$,
\begin{align}
 \norm{X(t)}_{\V} &\le C\mathfrak R_n^{1/2}e^{-\frac{13}{16}\sigma_n(t-t_n)} \norm{X(t_n)}_{\V},\label{eq:state-oneinterval}\\
 \|U_n(t)\|_{\R^{N_n}} &\le C\mathfrak R_n e^{-\frac{13}{16}\sigma_n(t-t_n)}\norm{X(t_n)}_{\V}, \label{eq:control-oneinterval}
\end{align}
where the control estimate follows from $\|U_n\|_{\R^{N_n}}\le\|R_nZ_n\|_{\R^{2N_n}}\le\mathfrak R_n^{1/2}\mathcal V_n^{1/2}$. In particular,
\begin{equation}\label{eq:endpoint-pre}
 \norm{X(t_{n+1})}_{\V}^2 \le C\mathfrak R_n e^{-\frac{13}{8}\sigma_n\ell_n} \norm{X(t_n)}_{\V}^2.
\end{equation}

We now choose $\sigma_n$ in terms of $\ell_n$ and obtain the following stabilization estimate on one interval.
\begin{proposition}\label{prop:oneinterval}
There exist constants $A_{\mathbb R},b_{\mathbb R},B_{\mathbb R},m_{\mathbb R}>0$ independent of $q$, $r_0$, and $n$. Set
$$\sigma_n=\frac{A_{\mathbb R}^2}{\ell_n^2}, \qquad \Lambda_n=Q+\frac{15}{16}\sigma_n.$$
If
\begin{equation}\label{eq:interval-smallness}
 \norm{X(t_n)}_{\V}\le e^{-B_{\mathbb R}/\ell_n},
\end{equation}
then the closed-loop solution with feedback \eqref{eq:physical-feedback} exists on $[t_n,t_{n+1}]$ and
\begin{align}
 \norm{X(t_{n+1})}_{\V} &\le e^{-b_{\mathbb R}/\ell_n}\norm{X(t_n)}_{\V}, \label{eq:interval-contraction}\\
 \norm{X(t)}_{\V} +\norm{\mathcal X_{\Dd_t}\mathcal K_n^{\mathbb R}(t)X(t)}_{L^2(\Omega;\mathbb R)} &\le e^{m_{\mathbb R}/\ell_n}\norm{X(t_n)}_{\V},\quad\text{a.e. }t\in I_n. \label{eq:interval-amplification}
\end{align}
Moreover, $\|\mathcal K_n^{\mathbb R}(t)\|_{\mathcal L(L^2(\Omega;\mathbb C),L^2(\Omega;\mathbb R))}\le e^{m_{\mathbb R}/\ell_n}$ after enlarging $m_{\mathbb R}$.
\end{proposition}

\begin{proof}
By \eqref{eq:Rn-bounds} and the choice $\Lambda_n=Q+15A_{\mathbb R}^2/(16\ell_n^2)$, we have
\begin{align*}
\log\mathfrak R_n &\le C+2\log C_0+3\log\frac1{\ell_n} +2C_0\sqrt{Q+\frac{15A_{\mathbb R}^2}{16\ell_n^2}} +\frac{3A_{\mathbb R}^2}{2\ell_n}+C|Q|\ell_n.
\end{align*}
For $0<\ell_n\le1$, $\log(1/\ell_n)\le1/\ell_n$, and there exists $C_Q>0$ such that
$$\sqrt{Q+\frac{15A_{\mathbb R}^2}{16\ell_n^2}} \le C_Q+\frac{A_{\mathbb R}}{\ell_n}.$$
It follows that there exist $C_1,C_2>0$ such that
\begin{equation}\label{eq:R-growth}
 \log\mathfrak R_n \le\frac1{\ell_n}\left(\frac32A_{\mathbb R}^2+C_1A_{\mathbb R}+C_2\right).
\end{equation}

Let $C_E$ be the fixed constant in \eqref{eq:endpoint-pre}. Choose $A_{\mathbb R}$ so large that $A_{\mathbb R}^2\ge\sigma_0$, $C_1A_{\mathbb R}+C_2\le A_{\mathbb R}^2/16$, and $\log C_E\le A_{\mathbb R}^2/32$. Then
\begin{equation}\label{eq:R-growth-simple}
 \log\mathfrak R_n\le\frac{25A_{\mathbb R}^2}{16\ell_n}.
\end{equation}
In particular, after increasing a fixed constant $C_R>0$ if necessary, $C_V\mathfrak R_n\le e^{C_R/\ell_n}$ for $n\ge 0$.

Since $\ell_n\le1$, it follows from the choice of $A_{\mathbb R}$ that $C_E\le e^{A_{\mathbb R}^2/(32\ell_n)}$. Using $\sigma_n\ell_n=A_{\mathbb R}^2/\ell_n$, \eqref{eq:endpoint-pre} and \eqref{eq:R-growth-simple}, we obtain
$$\norm{X(t_{n+1})}_{\V}^2 \le e^{-A_{\mathbb R}^2/(32\ell_n)}\norm{X(t_n)}_{\V}^2.$$
Therefore, \eqref{eq:interval-contraction} holds with $b_{\mathbb R}=A_{\mathbb R}^2/64$.

We next verify that \eqref{eq:interval-smallness} ensures the smallness required in the continuity argument. From \eqref{eq:eta-n}, \eqref{eq:R-growth}, and $\sigma_n=A_{\mathbb R}^2/\ell_n^2$, there is $C_\eta>0$ such that $\eta_n\ge e^{-C_\eta/\ell_n}.$ Thus, if \eqref{eq:interval-smallness} holds, we have
$$C_V\mathfrak R_n\norm{X(t_n)}_{\V}^2 \le e^{-(2B_{\mathbb R}-C_R)/\ell_n}.$$
Choose $B_{\mathbb R}$ so that $2B_{\mathbb R}-C_R\ge C_\eta+1$. Since $0<\ell_n\le1$, we have $e^{-1/\ell_n}\le e^{-1}<1/2$, and therefore \eqref{eq:bootstrap-initial} holds. The solution therefore exists, and the estimates \eqref{eq:state-oneinterval}--\eqref{eq:control-oneinterval} hold throughout $I_n$. Finally, \eqref{eq:R-growth} and \eqref{eq:state-oneinterval}--\eqref{eq:control-oneinterval}, together with $\norm{\mathcal X_{\Dd_t}\mathcal K_n^{\mathbb R}(t)X(t)}_{L^2(\Omega;\mathbb R)}\le\|U_n(t)\|_{\R^{N_n}}$, give \eqref{eq:interval-amplification}. Moreover, for $y=p+iq\in L^2(\Omega;\mathbb C)$, set $X=(p,q)^\top$. By \eqref{eq:physical-feedback} and the bound $\|B_n(t)\|_{\mathcal L(\R^{N_n},\R^{2N_n})}\le1$, we obtain
$$\begin{aligned}
\|\mathcal K_n^{\mathbb R}(t)y\|_{L^2(\Omega;\mathbb R)}
&=\|U_n(t,X)\|_{\R^{N_n}}
\le\|R_n(t-t_n)Z_n(X)\|_{\R^{2N_n}}\\
&\le\mathfrak R_n\|Z_n(X)\|_{\R^{2N_n}}
\le\mathfrak R_n\|y\|_{L^2(\Omega;\mathbb C)}.
\end{aligned}$$
This completes the proof.
\end{proof}

\section{Proofs of Theorems~\ref{thm:complex-control} and~\ref{thm:real-control}}\label{sec:iteration}

We apply the stabilization estimates successively on the intervals $I_n$. Fix $\mathbb K\in\{\mathbb C,\mathbb R\}$, with $\alpha\ne0$ in the real case. Note that in the real formulation, $\|X\|_{\V}=\|y\|_{H_0^1(\Omega;\mathbb C)}$. From Proposition~\ref{prop:complex-oneinterval} for $\mathbb K=\mathbb C$ and Proposition~\ref{prop:oneinterval} for $\mathbb K=\mathbb R$, there are constants $b_{\mathbb K},B_{\mathbb K},m_{\mathbb K}>0$, independent of $q$, $r_0$, and $n$, and finite-rank feedbacks $\mathcal K_n^{\mathbb K}(t)$ such that the condition on the initial datum
\begin{equation}\label{eq:common-oneinterval-small}
 \|y(t_n)\|_{H_0^1(\Omega;\mathbb C)}\le e^{-B_{\mathbb K}/\ell_n}
\end{equation}
ensures the existence of a unique solution of the closed-loop system on $[t_n,t_{n+1}]$ with
\begin{align}
\|y(t_{n+1})\|_{H_0^1(\Omega;\mathbb C)}&\le e^{-b_{\mathbb K}/\ell_n}\|y(t_n)\|_{H_0^1(\Omega;\mathbb C)},\label{eq:common-oneinterval-contract}\\
\|y(t)\|_{H_0^1(\Omega;\mathbb C)}+\|\mathcal X_{\Dd_t}\mathcal K_n^{\mathbb K}(t)y(t)\|_{L^2(\Omega;\mathbb K)}&\le e^{m_{\mathbb K}/\ell_n}\|y(t_n)\|_{H_0^1(\Omega;\mathbb C)},\quad\text{a.e. }t\in I_n.\label{eq:common-oneinterval-amp}
\end{align}
More precisely, the interval feedback for $\mathbb K=\mathbb C$ is $\mathcal X_E(t)(-\mu_{\Lambda_n}P_{\Lambda_n})$, with $\Lambda_n=A_{\mathbb C}^2/\ell_n^2$; for $\mathbb K=\mathbb R$ it is the real-valued operator defined in \eqref{eq:physical-feedback}, with $\sigma_n=A_{\mathbb R}^2/\ell_n^2$ and $\Lambda_n=Q+15\sigma_n/16$.

\begin{proof}[Proofs of Theorems~\ref{thm:complex-control} and~\ref{thm:real-control}]
Set $K_{\mathbb K}=1+\max\{B_{\mathbb K},m_{\mathbb K}\}$. Choose $q>1$ sufficiently close to $1$ so that
\begin{equation}\label{eq:common-q}
 K_{\mathbb K}(q-1)\le\frac12b_{\mathbb K}.
\end{equation}
For this choice of the ratio, take a left density point $t_*$ and the corresponding $r_0$ and intervals $I_n$ given by Lemma~\ref{lem:measurable-geometry}. In particular $\ell_n=\ell_0q^{-n}$, $t_n\uparrow t_*$, and $\ell_n\le1$. Define
\begin{equation}
\mathcal K_{\Dd}^{\mathbb K}(t)=\begin{cases}
0,&0<t<t_0,\\
\mathcal K_{\Lambda_n},&t\in I_n,\quad\mathbb K=\mathbb C,\\
\mathcal K_n^{\mathbb R}(t),&t\in I_n,\quad\mathbb K=\mathbb R,\\
0,&t\ge t_*.
\end{cases}
\end{equation}
We assign the feedback the value zero at the switching points. For $\mathbb K=\mathbb C$ we set $u(t)=\mathcal X_E(t)\mathcal K_{\Dd}^{\mathbb C}(t)y(t)$, and for $\mathbb K=\mathbb R$ we set $u(t)=\mathcal K_{\Dd}^{\mathbb R}(t)y(t)$. Thus $u(t)=\mathcal K_n^{\mathbb K}(t)y(t)$ on $I_n$. The family is strongly measurable and finite rank for each $t<t_*$.

We first assume that
\begin{equation}\label{eq:common-initial-node}
 \|y(t_0)\|_{H_0^1(\Omega;\mathbb C)}\le e^{-K_{\mathbb K}/\ell_0}.
\end{equation}
We claim that $\|y(t_n)\|_{H_0^1(\Omega;\mathbb C)}\le e^{-K_{\mathbb K}/\ell_n}$ for every $n\ge0$. The assertion holds for $n=0$. Assume that it holds for $n$. Since $K_{\mathbb K}>B_{\mathbb K}$, condition \eqref{eq:common-oneinterval-small} is satisfied, and
$$\|y(t_{n+1})\|_{H_0^1(\Omega;\mathbb C)}\le e^{-b_{\mathbb K}/\ell_n}e^{-K_{\mathbb K}/\ell_n}\le e^{-qK_{\mathbb K}/\ell_n}=e^{-K_{\mathbb K}/\ell_{n+1}},$$
by \eqref{eq:common-q}. The claim follows by induction.

Applying the endpoint estimates in \eqref{eq:common-oneinterval-contract} successively, we obtain
\begin{equation}
\|y(t_n)\|_{H_0^1(\Omega;\mathbb C)}\le\|y(t_0)\|_{H_0^1(\Omega;\mathbb C)}\exp\!\left[-\frac{b_{\mathbb K}(q^n-1)}{(q-1)\ell_0}\right].
\end{equation}
From \eqref{eq:common-oneinterval-amp}, for almost every $t\in I_n$,
\begin{equation}\label{eq:common-continuous-decay}
\begin{aligned}
\|y(t)\|_{H_0^1(\Omega;\mathbb C)}+\|\mathcal X_{\Dd_t}u(t)\|_{L^2(\Omega;\mathbb K)}&\le C_{\mathbb K,0}\exp\!\left[-\frac{c_{\mathbb K,0}q^n}{\ell_0}\right]\|y(t_0)\|_{H_0^1(\Omega;\mathbb C)},\end{aligned}
\end{equation}
where
$$C_{\mathbb K,0}=\exp\!\left(\frac{b_{\mathbb K}}{(q-1)\ell_0}\right),\qquad c_{\mathbb K,0}=\frac{b_{\mathbb K}}{q-1}-m_{\mathbb K}>0.$$

Since
$$\frac{\ell_n}{q-1}\le t_*-t\le\frac{q\ell_n}{q-1} \quad\text{for} \;t\in I_n,$$
it follows, with constants depending on $\mathbb K$ and the given set, that
\begin{equation}\label{eq:common-flat}
\|y(t)\|_{H_0^1(\Omega;\mathbb C)}+\|\mathcal X_{\Dd_t}u(t)\|_{L^2(\Omega;\mathbb K)}\le C\exp\!\left(-\frac{c}{t_*-t}\right)\|y(t_0)\|_{H_0^1(\Omega;\mathbb C)}.
\end{equation}
The first term tends to zero as $t\uparrow t_*$, and the solution therefore extends continuously by $y(t_*)=0$.

We extend $y$ and $u$ by zero on $[t_*,T]$. To verify the regularity of the strong solution across $t_*$, set
$$g(t)=\begin{cases}
	Qy(t)-(1+i\beta)|y(t)|^2y(t)+\mathcal X_{\Dd_t}u(t),&t_0<t<t_*,\\
	0,&t_*\le t\le T.
	\end{cases}$$
The bound \eqref{eq:common-continuous-decay} and the embedding $H_0^1(\Omega;\mathbb C)\hookrightarrow L^6(\Omega;\mathbb C)$ give $g\in L^2(t_0,T;L^2(\Omega;\mathbb C))$. Let $v$ solve
$$\partial_t v=(1+i\alpha)\Delta v+g,\qquad v|_{\partial\Omega}=0,\qquad v(t_0)=y(t_0).$$
The standard parabolic regularity estimate gives
$$v\in C([t_0,T];H_0^1(\Omega;\mathbb C))
\cap L^2(t_0,T;H^2(\Omega;\mathbb C)\cap H_0^1(\Omega;\mathbb C))
\cap H^1(t_0,T;L^2(\Omega;\mathbb C)),$$
and, by uniqueness for the linear equation, $v=y$ on every $[t_0,t_n]$. Continuity then gives $v(t_*)=0$. Since $g=0$ on $[t_*,T]$, we obtain $v=0=y$ on $[t_*,T]$. Thus the extended state has the asserted regularity. By local uniqueness, any other strong solution with the same value at $t_0$ agrees with $y$ on $[t_0,t_*)$. Its value at $t_*$ is therefore zero, and it remains zero on $[t_*,T]$ under the free evolution.

Let $z$ be the free solution given by Lemma~\ref{lem:appex-1} and set $y=z$ on $[0,t_0]$. If $\|y_0\|_{H_0^1(\Omega;\mathbb C)}\le\epsilon$, then $\|y(t_0)\|_{H_0^1(\Omega;\mathbb C)}\le C_f\|y_0\|_{H_0^1(\Omega;\mathbb C)}$. Thus the choice
$$\delta_{\mathbb K}=\min\left\{\epsilon,\,C_f^{-1}e^{-K_{\mathbb K}/\ell_0}\right\}$$
ensures \eqref{eq:common-initial-node}. The free solution on $[0,t_0]$ and the controlled solution on $[t_0,T]$ agree at $t_0$. Their concatenation therefore belongs to the same class of strong solutions on $[0,T]$. Combining Lemma~\ref{lem:appex-1} with \eqref{eq:common-flat}, we may choose the constants in Theorems~\ref{thm:complex-control} and~\ref{thm:real-control} as $M_{\mathbb K}=C_f C_{\mathbb K,0}$ and $c_{\mathbb K}=c_{\mathbb K,0}/(q-1)$. This completes the proof.
\end{proof}

\begin{proof}[Proof of Remark~\ref{rem:cylinder-both}]
Fix $\mathbb K\in\{\mathbb C,\mathbb R\}$, with $\alpha\ne0$ in the real case. Choose $0<\rho<1$ and $x_0\in\Omega$ such that $B_{4\rho}(x_0)\subset\Omega$ and $\omega_0=\omega\cap B_\rho(x_0)$ has positive measure. Applying Lemma~\ref{lem:spectral-measurable} with $\nu=|\omega_0|/2$, we obtain the spectral inequality on $\omega_0$ and hence on $\omega$, with a constant $C_0$ independent of $T$. Since $E=(0,T)$, the constants in Propositions~\ref{prop:complex-oneinterval} and~\ref{prop:oneinterval}, and the corresponding $K_{\mathbb K}$ and $q$ in \eqref{eq:common-q}, can be chosen independently of $0<T<1$.

For these fixed parameters, set
$$t_*=\frac{3T}{4},\qquad r_0=\frac{T}{2},\qquad t_0=\frac{T}{4},\qquad \ell_0=\frac{T(q-1)}{2q}.$$
All the intervals $I_n$ lie in $E$, so \eqref{eq:active-fraction} holds. The parameter choices at the end of the preceding proof give
$$\begin{aligned}
\delta_{\mathbb K}(T)&=\min\left\{\epsilon,\,C_f^{-1}\exp\!\left(-\frac{2qK_{\mathbb K}}{(q-1)T}\right)\right\},\\
M_{\mathbb K}(T)&=C_f\exp\!\left(\frac{2qb_{\mathbb K}}{(q-1)^2T}\right).
\end{aligned}$$
Here $\epsilon$ and $C_f$ are independent of $T$ by Lemma~\ref{lem:appex-1}. Thus both estimates in \eqref{eq:cylinder-cost} hold after choosing $C_{\mathbb K}$ sufficiently large, independently of $T$.
\end{proof}

\begin{proof}[Proof of Remark~\ref{rem:alpha-zero}]
Suppose, to the contrary, that the estimate \eqref{eq:impossible-linear-cost} holds for every sufficiently small initial datum. Fix a nonzero real $q_0\in H_0^1(\Om;\mathbb R)$ and take
$$y_0^\varepsilon=i\varepsilon q_0, \qquad0<\varepsilon\ll1.$$
Let $u^\varepsilon$ be a real control driving the corresponding solution to zero at time $T$, with
\begin{equation}\label{eq:eps-control}
		\norm{\mathcal X_{\Dd}u^\varepsilon}_{L^\infty(0,T;L^2(\Omega;\mathbb R))} \le C\varepsilon\norm{q_0}_{H_0^1(\Omega;\mathbb R)}.
	\end{equation}
By \eqref{eq:eps-control}, the initial datum and the forcing $f=\mathcal X_{\Dd}u^\varepsilon$ satisfy the smallness condition in Lemma~\ref{lem:appex-1} for $\varepsilon$ sufficiently small. Hence there exists $C_T>0$ such that
\begin{equation}\label{eq:eps-state}
\|y^\varepsilon\|_{C([0,T];H_0^1(\Omega;\mathbb C))}\le C_T\varepsilon.
\end{equation}

Set $z^\varepsilon=y^\varepsilon/\varepsilon=p^\varepsilon+iq^\varepsilon$. From \eqref{eq:eps-state}, $z^\varepsilon$ is uniformly bounded in $C([0,T];H_0^1(\Omega;\mathbb C))$. Since $\alpha=0$ and the scaled control is real, the imaginary part satisfies the Dirichlet problem
\begin{equation}\label{eq:scaled-imaginary}
		\partial_t q^\varepsilon =\Delta q^\varepsilon+Qq^\varepsilon -\varepsilon^2\operatorname{Im}\bigl((1+i\beta)|z^\varepsilon|^2z^\varepsilon\bigr), \qquad q^\varepsilon(0)=q_0.
	\end{equation}
By \eqref{eq:F-cubic} and \eqref{eq:eps-state}, the nonlinear term in \eqref{eq:scaled-imaginary} is $O(\varepsilon^2)$ in $L^\infty(0,T;L^2(\Omega;\mathbb R))$, and hence
\begin{equation}\label{eq:imaginary-limit}
		q^\varepsilon(T)=e^{T(\Delta_D+Q)}q_0+O_{L^2(\Omega;\mathbb R)}(\varepsilon^2).
	\end{equation}
Note that $e^{T(\Delta_D+Q)}q_0\ne0$ when $q_0\ne0$. On the other hand, $y^\varepsilon(T)=0$ implies $q^\varepsilon(T)=0$. Letting $\varepsilon\downarrow0$ in \eqref{eq:imaginary-limit} gives a contradiction.
\end{proof}

\section{Proof of Theorem~\ref{thm:minimal-time}}\label{sec:minimal-time}
Throughout this section, we set $\mathcal D=(0,s)\times\omega$ and $F(z)=|z|^2z$.

\begin{lemma}\label{lem:cubic-difference}
For every $\beta\in\mathbb R$ there exists $C_\beta>0$ such that for all $a,b\in\mathbb C$,
\begin{equation}
\operatorname{Re}\bigl((1+i\beta)(F(a+b)-F(a))\overline b\bigr)\ge\frac12|b|^4-C_\beta|a|^2|b|^2.
\end{equation}
\end{lemma}

\begin{proof}
A direct expansion gives $F(a+b)-F(a)=|b|^2b+\mathcal R(a,b)$ with $|\mathcal R(a,b)|\le C(|a||b|^2+|a|^2|b|)$. Hence
$$\operatorname{Re}\bigl((1+i\beta)(F(a+b)-F(a))\overline b\bigr)\ge |b|^4-C(1+|\beta|)(|a||b|^3+|a|^2|b|^2).$$
Using Young's inequality, we have $|a||b|^3\le\varepsilon |b|^4+C_\varepsilon|a|^2|b|^2$. The claim follows by choosing $\varepsilon$ sufficiently small.
\end{proof}

\begin{lemma}\label{lem:shielding}
Assume \eqref{eq:fixed-hole}. There are $t_{\rm sh}\in(0,1)$ and $c,C>0$, depending only on $\Omega,r_*,\alpha,\beta,Q$, with the following property. For any $0<s\le t_{\rm sh}$, let $y,z$ be strong solutions on $[0,s]$ with the same initial datum, where $y$ is controlled and $z$ is a free solution satisfying $\|z\|_{L^\infty((0,s)\times\Omega;\mathbb C)}\le1$. Then
\begin{equation}\label{eq:shielding}
\|y(t)-z(t)\|_{L^2(B_{2r_*}(x_*);\mathbb C)}\le Ce^{-c/t},\quad \forall 0<t\le s.
\end{equation}
The constants are uniform over all such $y,z$ and controls.
\end{lemma}

\begin{proof}
We first use the dissipation of the cubic term to obtain a local energy estimate which is independent of the control. A related estimate for the cubic heat equation can be found in \cite[Remark~7.7]{Coron}. We then apply this estimate successively with cutoffs supported in balls of decreasing radii to obtain \eqref{eq:shielding}.

Put $w=y-z$. Then $w(0)=0$ and
\begin{equation}\label{eq:diff-w}
\partial_t w=(1+i\alpha)\Delta w+Qw-(1+i\beta)(F(z+w)-F(z))+\mathcal X_\omega u.
\end{equation}
Using Lemma~\ref{lem:cubic-difference} and the fact that $|z|\le1$, we obtain
$$\operatorname{Re}\bigl((1+i\beta)(F(z+w)-F(z))\overline w\bigr)\ge\frac12|w|^4-C_\beta|w|^2.$$
Choose $\psi\in C_c^\infty(B_{4r_*}(x_*))$ with $0\le\psi\le1$ and $\psi=1$ on $B_{3r_*}(x_*)$. Put $\chi_0=\psi^2$ and set
$$E_0(t)=\int_\Omega\chi_0^2(x)|w(t,x)|^2\,\mathrm dx.$$
Note that $\chi_0\mathcal X_\omega=0$ almost everywhere. 

Testing \eqref{eq:diff-w} against $\chi_0^2\overline w$, taking the real part, and absorbing the mixed gradient term, we obtain
\begin{equation}\label{eq:E0-energy}
\frac12E_0'(t)+\frac12\int\chi_0^2|\nabla w|^2+\frac12\int\chi_0^2|w|^4\le (|Q|+C_\beta)E_0(t)+C_\alpha\int|\nabla\chi_0|^2|w|^2.
\end{equation}
Since $\chi_0=\psi^2$, we have $|\nabla\chi_0|^2\le4\|\nabla\psi\|_{L^\infty(\Omega;\mathbb R^d)}^2\chi_0$. Moreover, the pointwise bound $\chi_0|w|^2\le\varepsilon\chi_0^2|w|^4+C_\varepsilon$ holds for every $\varepsilon>0$. Thus \eqref{eq:E0-energy} yields
\begin{equation}
E_0'(t)\le A E_0(t)+B,\qquad E_0(0)=0.
\end{equation}
Therefore, after restricting $t_{\rm sh}\le1$, there exists $C_{\rm sh}>0$ such that
\begin{equation}\label{eq:E0-small}
\sup_{0\le r\le t}E_0(r)\le C_{\rm sh}t,\quad \forall 0<t\le s.
\end{equation}
All the constants obtained so far are independent of the controlled trajectory and the norm of the control.

Next let $N\ge1$ be an integer, $\delta=r_*/(N+1)$ and $r_j=3r_*-j\delta$ for $j=0,\ldots,N$. Choose $\chi_j\in C_c^\infty(B_{r_{j-1}}(x_*))$, $1\le j\le N$, such that $0\le\chi_j\le1$, $\chi_j=1$ on $B_{r_j}(x_*)$ and
\begin{equation}
\|\nabla\chi_j\|_{L^\infty(\Omega;\mathbb R^d)}\le C\frac{N+1}{r_*},\qquad \operatorname{supp}\nabla\chi_j\subset\{\chi_{j-1}=1\},
\end{equation}
where $\chi_0=1$ on $B_{3r_*}(x_*)$. Define $E_j(t)=\int_\Omega\chi_j^2(x)|w(t,x)|^2\,\mathrm dx$. A localized energy estimate as above, with the nonnegative quartic term discarded, yields
$$E_j'(t)\le A E_j(t)+D(N+1)^2 E_{j-1}(t),\qquad E_j(0)=0.$$
Here $A,D$ are independent of $j,N,y,z,u$. It follows from Gronwall's inequality that
\begin{equation}
E_j(t)\le D(N+1)^2\int_0^t e^{A(t-s)}E_{j-1}(s)\,\mathrm ds.
\end{equation}
Iterating this estimate $N$ times and using \eqref{eq:E0-small}, we find
\begin{equation}\label{eq:EN-bound}
E_N(t)\le C_{\rm sh}t e^{At}\frac{[D(N+1)^2t]^N}{N!}\le C_{\rm sh}t e^{At}(4eDNt)^N,
\end{equation}
where we used $N!\ge(N/e)^N$ and $(N+1)^2\le4N^2$.

Choose $a>0$ sufficiently small so that $4eDa\le1/2$, and decrease $t_{\rm sh}$ if necessary so that $t_{\rm sh}\le a/2$. For $0<t\le s$, let $N=\lfloor a/t\rfloor\ge a/(2t)$. Then \eqref{eq:EN-bound} gives
$$E_N(t)\le Ct2^{-N}\le Ct\exp\!\left(-\frac{a\log2}{2t}\right).$$
Since $r_N=2r_*+\delta$ and $\chi_N=1$ on $B_{r_N}(x_*)$, we have $\|w(t)\|_{L^2(B_{2r_*}(x_*);\mathbb C)}^2\le E_N(t)$. Adjusting the constants $c,C$, we obtain \eqref{eq:shielding}.
\end{proof}
\begin{proof}[Proof of Theorem~\ref{thm:minimal-time}]
Choose a nonzero real-valued $\phi\in C_c^\infty(B_{r_*}(x_*);\mathbb R)$ with $\|\phi\|_{H_0^1(\Omega;\mathbb R)}=1$. For $R$ sufficiently small, Lemma~\ref{lem:appex-1} gives a free solution $z_R$ starting from $R\phi$. By \eqref{eq:small-data-H2}, there exists a constant $C>0$, independent of $R$, such that
$$
	\|z_R\|_{C([0,1];H^2(\Omega;\mathbb C))}
	+\|\partial_t z_R\|_{C([0,1];L^2(\Omega;\mathbb C))}
	\le CR.
	$$
The Sobolev embedding $H^2(\Omega;\mathbb C)\hookrightarrow L^\infty(\Omega;\mathbb C)$ (see \cite{Brezis}) gives $\|z_R\|_{L^\infty((0,1)\times\Omega;\mathbb C)}\le CR$, after increasing $C$ if necessary. Since $z_R(0)=R\phi$, it also follows that
$$
	\|z_R(t)-R\phi\|_{L^2(\Omega;\mathbb C)}
	\le\int_0^t
	\|\partial_s z_R(s)\|_{L^2(\Omega;\mathbb C)}\,\mathrm ds
	\le CRt,\quad \forall 0\le t\le1.
	$$

Thus there exist $R_1\in(0,1)$ and $C_1>0$ such that, for $0<R<R_1$,
\begin{equation}\label{eq:uncontrolled-small}
			\|z_R\|_{L^\infty((0,1)\times\Omega;\mathbb C)}
			\le C_1R,\qquad
			\|z_R(t)-R\phi\|_{L^2(\Omega;\mathbb C)}
			\le C_1Rt,\quad 0\le t\le1.
		\end{equation}
Reducing $R_1$ if necessary, we may assume $C_1R_1\le1$. Set
$$
		t_1=\min\left\{
		1,\frac{\|\phi\|_{L^2(\Omega;\mathbb R)}}{2C_1}
		\right\}.
		$$

Suppose that a complex control drives $R\phi$ to zero at a time $0<s\le\min\{t_{\rm sh},t_1\}$, and denote the controlled solution by $y$. Since $y(s)=0$, Lemma~\ref{lem:shielding} gives
$$
		\|z_R(s)\|_{L^2(B_{2r_*}(x_*);\mathbb C)}
		=\|z_R(s)-y(s)\|_{L^2(B_{2r_*}(x_*);\mathbb C)}
		\le Ce^{-c/s}.
		$$
On the other hand, from \eqref{eq:uncontrolled-small} and $\operatorname{supp}\phi\subset B_{r_*}(x_*)$, we have
$$
		\|z_R(s)\|_{L^2(B_{2r_*}(x_*);\mathbb C)}
		\ge R\|\phi\|_{L^2(\Omega;\mathbb R)}-C_1Rs
		\ge \frac{R}{2}\|\phi\|_{L^2(\Omega;\mathbb R)},
		$$
where the last inequality follows from the choice of $t_1$. Hence $R\le C'e^{-c/s}$ for some $C'\ge1$, and therefore $ s\ge\frac{c}{\log(C'/R)}. $ For $R_1$ sufficiently small, $c/\log(C'/R)\le\min\{t_{\rm sh},t_1\}$ whenever $0<R<R_1$, so the same lower bound also holds for larger $s$. This proves the lower bound for complex controls, and hence also for real controls, because every real control is also an admissible complex control.

We now prove the upper bound. Fix $\mathbb K\in\{\mathbb C,\mathbb R\}$, with $\alpha\ne0$ in the real case. From \eqref{eq:cylinder-cost}, there exist $b,B>0$ and $s_0\in(0,1)$ such that every initial datum satisfying $\|y_0\|_{H_0^1(\Omega;\mathbb C)}\le be^{-B/s}$ can be driven to zero in time $s$ by an admissible $\mathbb K$-valued control whenever $0<s\le s_0$. For $0<R<b$ sufficiently small, set $s_R=2B/\log(b/R)\le s_0$. Then $be^{-B/s_R}=\sqrt{bR}\ge R$, so $\mathcal T_{\mathbb K}(R)\le s_R\le C/|\log R|$. Together with the lower bound, this proves \eqref{eq:tau-two-sided}.
\end{proof}


{\small \begin{thebibliography}{99}

\bibitem{ASK} O.~M. Aamo, A. Smyshlyaev, M. Krstic, Boundary control of the linearized Ginzburg--Landau model of vortex shedding, \emph{SIAM J. Control Optim.} \textbf{43} (2005), 1953--1971.

\bibitem{ABDG} F. Ammar Khodja, A. Benabdallah, C. Dupaix, M. Gonz\'alez-Burgos, A Kalman rank condition for the localized distributed controllability of a class of linear parabolic systems, \emph{J. Evol. Equ.} \textbf{9} (2009), 267--291.

\bibitem{ABDK} F. Ammar Khodja, A. Benabdallah, C. Dupaix, I. Kostin, Controllability to the trajectories of phase-field models by one control force, \emph{SIAM J. Control Optim.} \textbf{42} (2003), 1661--1680.

\bibitem{AnitaTataru} S. Ani\c{t}a, D. Tataru, Null controllability for the dissipative semilinear heat equation, \emph{Appl. Math. Optim.} \textbf{46} (2002), 97--105.

\bibitem{AE2013} J. Apraiz, L. Escauriaza, Null-control and measurable sets, \emph{ESAIM Control Optim. Calc. Var.} \textbf{19} (2013), 239--254.

\bibitem{AEWZ} J. Apraiz, L. Escauriaza, G. Wang, C. Zhang, Observability inequalities and measurable sets, \emph{J. Eur. Math. Soc.} \textbf{16} (2014), 2433--2475.

\bibitem{Brezis} H. Br\'ezis, \emph{Functional Analysis, Sobolev Spaces and Partial Differential Equations}, Universitext, Springer, New York, 2011.

\bibitem{CMM} N. Carre\~no, A. Mercado, R. Morales, Local null controllability of a cubic Ginzburg--Landau equation with dynamic boundary conditions, \emph{J. Evol. Equ.} \textbf{25} (2025), Paper No.~62, 31 pp.

\bibitem{CLB} X. Chen, K. Le Balc'h, Observability from measurable sets for strongly coupled parabolic systems, arXiv:2610.31876, 2026.

\bibitem{Coron} J.-M. Coron, \emph{Control and Nonlinearity}, Mathematical Surveys and Monographs, Vol.~136, American Mathematical Society, Providence, RI, 2007.

\bibitem{CN} J.-M. Coron, H.-M. Nguyen, Null controllability and finite time stabilization for the heat equations with variable coefficients in space in one dimension via backstepping approach, \emph{Arch. Ration. Mech. Anal.} \textbf{225} (2017), 993--1023.

\bibitem{DeNicolaoStrada1997} G. De Nicolao, S. Strada, On the use of reachability Gramians for the stabilization of linear periodic systems, \emph{Automatica} \textbf{33} (1997), 729--732.

\bibitem{EMZ2015}
L.~Escauriaza, S.~Montaner, C.~Zhang,
Observation from measurable sets for parabolic analytic evolutions and applications,
\emph{J. Math. Pures Appl.} \textbf{104} (2015), 837--867.

\bibitem{EMZ} L. Escauriaza, S. Montaner, C. Zhang, Analyticity of solutions to parabolic evolutions and applications, \emph{SIAM J. Math. Anal.} \textbf{49} (2017), 4064--4092.

\bibitem{FGT} E. Fern\'andez-Cara, M. Gonz\'alez-Burgos, L. de Teresa, Controllability of linear and semilinear non-diagonalizable parabolic systems, \emph{ESAIM Control Optim. Calc. Var.} \textbf{21} (2015), 1178--1204.

\bibitem{FernandezCaraZuazua} E. Fern\'andez-Cara, E. Zuazua, Null and approximate controllability for weakly blowing up semilinear heat equations, \emph{Ann. Inst. H. Poincar\'e C Anal. Non Lin\'eaire} \textbf{17} (2000), 583--616.

\bibitem{F} X. Fu, Null controllability for the parabolic equation with a complex principal part, \emph{J. Funct. Anal.} \textbf{257} (2009), 1333--1354.

\bibitem{FWYZ} X. Fu, G. Wang, H. Yu, X. Zhu, Observability from measurable sets for strongly coupled parabolic systems via single-component observation, arXiv:2604.13599, 2026.

\bibitem{GHMSXZ} L. Gagnon, A. Hayat, S. Marx, S. Xiang, C. Zhang, Quantitative Fredholm backstepping and rapid stabilization, arXiv:2605.17941, 2026.

\bibitem{GBPG} M. Gonz\'alez-Burgos, R. P\'erez-Garc\'ia, Controllability of some coupled parabolic systems by one control force, \emph{C. R. Math. Acad. Sci. Paris} \textbf{340} (2005), 125--130.

\bibitem{Guerrero} S. Guerrero, Null controllability of some systems of two parabolic equations with one control force, \emph{SIAM J. Control Optim.} \textbf{46} (2007), 379--394.

\bibitem{KLZ2026} M. Kumar, X. Liu, X. Zhang, Quantitative linear approximation for controllability of quasi-linear parabolic equations, arXiv:2609.04969, 2026.

\bibitem{LR} G. Lebeau, L. Robbiano, Contr\^ole exact de l'\'equation de la chaleur, \emph{Comm. Partial Differential Equations} \textbf{20} (1995), 335--356.

\bibitem{LZ} G. Lebeau, E. Zuazua, Null-controllability of a system of linear thermoelasticity, \emph{Arch. Ration. Mech. Anal.} \textbf{141} (1998), 297--329.

\bibitem{LissyZuazua} P. Lissy, E. Zuazua, Internal observability for coupled systems of linear partial differential equations, \emph{SIAM J. Control Optim.} \textbf{57} (2019), 832--853.

\bibitem{LZ2} S. Luan, X. Zhou, Approximation of time optimal controls for Ginzburg--Landau equations with disturbance coefficients, \emph{SIAM J. Control Optim.} \textbf{62} (2024), 3195--3212.

\bibitem{NguyenSch} H.-M. Nguyen, Rapid stabilization and finite time stabilization of the bilinear Schr\"odinger equation, arXiv:2405.10002, 2024.

\bibitem{NguyenKdV} H.-M. Nguyen, Rapid and finite-time boundary stabilization of a KdV system, \emph{SIAM J. Control Optim.} \textbf{64} (2026), 76--101.

\bibitem{PW} K.~D. Phung, G. Wang, An observability estimate for parabolic equations from a measurable set in time and its applications, \emph{J. Eur. Math. Soc.} \textbf{15} (2013), 681--703.

\bibitem{PWZ} K.~D. Phung, L.~J. Wang, C. Zhang, Bang-bang property for time optimal control of semilinear heat equation, \emph{Ann. Inst. H. Poincar\'e C Anal. Non Lin\'eaire} \textbf{31} (2014), 477--499.

\bibitem{PorrettaZuazua} A. Porretta, E. Zuazua, Null controllability of viscous Hamilton--Jacobi equations, \emph{Ann. Inst. H. Poincar\'e C Anal. Non Lin\'eaire} \textbf{29} (2012), 301--333.

\bibitem{RZ} L. Rosier, B.~Y. Zhang, Null controllability of the complex Ginzburg--Landau equation, \emph{Ann. Inst. H. Poincar\'e C Anal. Non Lin\'eaire} \textbf{26} (2009), 649--673.

\bibitem{W} G. Wang, $L^\infty$-null controllability for the heat equation and its consequences for the time optimal control problem, \emph{SIAM J. Control Optim.} \textbf{47} (2008), 1701--1720.

\bibitem{WangZhang2026}
M.~Wang and C.~Zhang,
\newblock Observability from measurable sets for heat equations with bounded space--time potentials,
\newblock arXiv:2611.16811, 2026.

\bibitem{X1}
S. Xiang, Small-time local stabilization of the two-dimensional incompressible Navier--Stokes equations, \emph{Ann. Inst. H. Poincar\'e C Anal. Non Lin\'eaire} \textbf{40} (2023), 1487--1511.

\bibitem{X2} S. Xiang, Quantitative rapid and finite time stabilization of the heat equation, \emph{ESAIM Control Optim. Calc. Var.} \textbf{30} (2024), Paper No.~40, 25 pp.

\bibitem{XXZLQ} S. Xiang, Y. Xiao, C. Zhang, Quantitative rapid stabilization for parabolic equations via the linear quadratic theory, arXiv:2606.09063, 2026.

\bibitem{XXZMD} S. Xiang, Y. Xiao, C. Zhang, Quantitative rapid boundary stabilization via modal decomposition and its application to the Allen--Cahn equation, arXiv:2607.03031, 2026.

\bibitem{XZ} Y. Xiao, C. Zhang, Rapid boundary stabilization of one-dimensional nonlinear parabolic equations, \emph{SIAM J. Control Optim.} \textbf{64} (2026), 3967--3988.

\bibitem{ZZ} X. Zeng, C. Zhang, Small-time approximate controllability for the nonlinear complex Ginzburg--Landau equation with bilinear control, arXiv:2512.22512, 2025.

\end{thebibliography}}
\end{document}